\documentclass[12pt]{amsart}      %{article} was 12pt latex e
\usepackage{txfonts}
\usepackage{amssymb}
\usepackage{eucal}
\usepackage{graphicx}
\usepackage{amsmath}
\usepackage{amscd}
\usepackage[all]{xy}           %xypic macro for latex2.09
\usepackage[active]{srcltx} %SRC Specials for DVI Searching

\usepackage{amsfonts,latexsym}
\usepackage{xspace}
\usepackage{epsfig}
\usepackage{float}
\usepackage{color}
\usepackage{fancybox}
\usepackage{colordvi}
\usepackage{multicol}
\usepackage{colordvi}
\usepackage[hypertex]{hyperref} %put this package last

\newtheorem{theorem}{Theorem}[section]
\newtheorem{prop}[theorem]{Proposition}
\newtheorem{defn}[theorem]{Definition}
\newtheorem{lemma}[theorem]{Lemma}
\newtheorem{coro}[theorem]{Corollary}
\newtheorem{prop-def}{Proposition-Definition}[section]
\newtheorem{coro-def}{Corollary-Definition}[section]

\newtheorem{remark}[theorem]{Remark}

\newcommand{\nc}{\newcommand}
\nc{\tred}[1]{\textcolor{red}{#1}}
\nc{\tblue}[1]{\textcolor{blue}{#1}}
\nc{\tgreen}[1]{\textcolor{green}{#1}}
\nc{\tpurple}[1]{\textcolor{purple}{#1}}
\nc{\btred}[1]{\textcolor{red}{\bf #1}}
\nc{\btblue}[1]{\textcolor{blue}{\bf #1}}
\nc{\btgreen}[1]{\textcolor{green}{\bf #1}}
\nc{\btpurple}[1]{\textcolor{purple}{\bf #1}}

\renewcommand{\Bbb}{\mathbb}
\renewcommand{\frak}{\mathfrak}
\newcommand{\efootnote}[1]{}
\renewcommand{\textbf}[1]{}
\newcommand{\delete}[1]{}
\nc{\dfootnote}[1]{{}}          %{{}}
\nc{\ffootnote}[1]{\dfootnote{#1}}
\nc{\mfootnote}[1]{\footnote{#1}} % Use this to show footnotes
\nc{\ofootnote}[1]{\footnote{\tiny Older version: #1}}
\nc{\mlabel}[1]{\label{#1}}  % Use this to suppress names
\nc{\mcite}[1]{\cite{#1}}  % Use this to suppress names
\nc{\mref}[1]{\ref{#1}}  % Use this to suppress names
\nc{\mbibitem}[1]{\bibitem{#1}} % Use this to show number
\delete{
\nc{\mlabel}[1]{\label{#1}  % Use the next two lines to show names
{\hfill \hspace{1cm}{\bf{{\ }\hfill(#1)}}}}
\nc{\mcite}[1]{\cite{#1}{{\bf{{\ }(#1)}}}}  % Use this lines to show names
\nc{\mref}[1]{\ref{#1}{{\bf{{\ }(#1)}}}}  % Use this lines to show names
\nc{\mbibitem}[1]{\bibitem[\bf #1]{#1}} % Use this to show name
}

\nc{\mtail}{\leq_t}
\nc{\mhead}{\leq_h}
\nc{\rk}{\mathrm{rk}}
\nc{\mset}[1]{\tilde{#1}}
\nc{\pa}{\frakL}
\nc{\arr}{\rightarrow}
\nc{\lu}[1]{(#1)}
\nc{\mult}{\mrm{mult}}
\nc{\diff}{\mathrm{Der}}
\nc{\indiff}{\mathrm{InDer}}
\nc{\outdiff}{\mathrm{OutDer}}
\nc{\conmat}{connection matrix\xspace}
\nc{\bounmat}{boundary matrix\xspace}
\nc{\pcyc}{\mathfrak c}
\nc{\calpa}{\calp_A}
\nc{\calpal}{\Gamma_{AL}}
\nc{\calpc}{\calp_L}
\nc{\frakDa}{\frakD_1}
\nc{\frakDal}{\frakD_2}
\nc{\frakDc}{\frakD_L}
\nc{\frakDv}{\frakD_V}
\nc{\frakDp}{\frakD_F}
\nc{\frakBa}{\frakB_1}
\nc{\frakBal}{\frakB_2}
\nc{\frakBc}{\frakB_L}
\nc{\frakBv}{\frakB_V}

\nc{\bin}[2]{ (_{\stackrel{\scs{#1}}{\scs{#2}}})}  %binomial coeff
\nc{\binc}[2]{ \left (\!\! \begin{array}{c} \scs{#1}\\
    \scs{#2} \end{array}\!\! \right )}  %binomial coeff
\nc{\bincc}[2]{  \left ( {\scs{#1} \atop
    \vspace{-1cm}\scs{#2}} \right )}  %binomial coeff
\nc{\bs}{\bar{S}}
\nc{\cosum}{\sqsubset}
\nc{\la}{\longrightarrow}
\nc{\rar}{\rightarrow}
\nc{\dar}{\downarrow}
\nc{\dprod}{**}
\nc{\dap}[1]{\downarrow \rlap{$\scriptstyle{#1}$}}
\nc{\md}{\mathrm{dth}}
\nc{\uap}[1]{\uparrow \rlap{$\scriptstyle{#1}$}}
\nc{\defeq}{\stackrel{\rm def}{=}}
\nc{\disp}[1]{\displaystyle{#1}}
\nc{\dotcup}{\ \displaystyle{\bigcup^\bullet}\ }
\nc{\gzeta}{\bar{\zeta}}
\nc{\hcm}{\ \hat{,}\ }
\nc{\hts}{\hat{\otimes}}
\nc{\barot}{{\otimes}}
\nc{\free}[1]{\bar{#1}}
\nc{\uni}[1]{\tilde{#1}}
\nc{\hcirc}{\hat{\circ}}
\nc{\lleft}{[}
\nc{\lright}{]}
\nc{\lc}{\lfloor}
\nc{\rc}{\rfloor}
\nc{\curlyl}{\left \{ \begin{array}{c} {} \\ {} \end{array}
    \right .  \!\!\!\!\!\!\!}
\nc{\curlyr}{ \!\!\!\!\!\!\!
    \left . \begin{array}{c} {} \\ {} \end{array}
    \right \} }
\nc{\longmid}{\left | \begin{array}{c} {} \\ {} \end{array}
    \right . \!\!\!\!\!\!\!}
\nc{\onetree}{\bullet}
\nc{\ora}[1]{\stackrel{#1}{\rar}}
\nc{\ola}[1]{\stackrel{#1}{\la}}%${\Bbb Z}$
\nc{\ot}{\otimes}
\nc{\mot}{{{\boxtimes\,}}}
\nc{\otm}{\overline{\boxtimes}}
\nc{\sprod}{\bullet}
\nc{\scs}[1]{\scriptstyle{#1}}
\nc{\mrm}[1]{{\rm #1}}
\nc{\margin}[1]{\marginpar{\rm #1}}   %{\rm #1}}
\nc{\dirlim}{\displaystyle{\lim_{\longrightarrow}}\,}
\nc{\invlim}{\displaystyle{\lim_{\longleftarrow}}\,}
\nc{\mvp}{\vspace{0.3cm}}
\nc{\tk}{^{(k)}}
\nc{\tp}{^\prime}
\nc{\ttp}{^{\prime\prime}}
\nc{\svp}{\vspace{2cm}}
\nc{\vp}{\vspace{8cm}}
\nc{\proofbegin}{\noindent{\bf Proof: }}
\nc{\proofend}{$\blacksquare$ \vspace{0.3cm}}
\nc{\modg}[1]{\!<\!\!{#1}\!\!>}
\nc{\intg}[1]{F_C(#1)}
\nc{\lmodg}{\!<\!\!}
\nc{\rmodg}{\!\!>\!}
\nc{\cpi}{\widehat{\Pi}}
\nc{\sha}{{\mbox{\cyr X}}}  %used to be \cyr
\nc{\shap}{{\mbox{\cyrs X}}} %sha as product
\nc{\shpr}{\diamond}    %Shuffle product
\nc{\shp}{\ast}
\nc{\shplus}{\shpr^+}
\nc{\shprc}{\shpr_c}    %Cartier's product
\nc{\msh}{\ast}
\nc{\zprod}{m_0}
\nc{\oprod}{m_1}
\nc{\vep}{\varepsilon}
\nc{\labs}{\mid\!}
\nc{\rabs}{\!\mid}
\nc{\mmbox}[1]{\mbox{\ #1\ }}
\nc{\fp}{\mrm{FP}} \nc{\rchar}{\mrm{char}} \nc{\End}{\mrm{End}} \nc{\Fil}{\mrm{Fil}}
\nc{\Mor}{Mor\xspace}
\nc{\gmzvs}{gMZV\xspace}
\nc{\gmzv}{gMZV\xspace}
\nc{\mzv}{MZV\xspace}
\nc{\mzvs}{MZVs\xspace}
\nc{\Hom}{\mrm{Hom}} \nc{\id}{\mrm{id}} \nc{\im}{\mrm{im}}
\nc{\incl}{\mrm{incl}} \nc{\map}{\mrm{Map}} \nc{\mchar}{\rm char}
\nc{\nz}{\rm NZ} \nc{\supp}{\mathrm Supp}

\nc{\Alg}{\mathbf{Alg}}
\nc{\Bax}{\mathbf{Bax}}
\nc{\bff}{\mathbf f}
\nc{\bfk}{{\bf k}}
\nc{\bfone}{{\bf 1}}
\nc{\bfx}{\mathbf x}
\nc{\bfy}{\mathbf y}
\nc{\base}[1]{\bfone^{\otimes ({#1}+1)}} %{{a_{#1}}}
\nc{\Cat}{\mathbf{Cat}}

\nc{\detail}{\marginpar{\bf More detail}
    \noindent{\bf Need more detail!}
    \svp}
\nc{\Int}{\mathbf{Int}}
\nc{\Mon}{\mathbf{Mon}}
\nc{\rbtm}{{shuffle }}
\nc{\rbto}{{Rota-Baxter }}
\nc{\remarks}{\noindent{\bf Remarks: }}
\nc{\Rings}{\mathbf{Rings}}
\nc{\Sets}{\mathbf{Sets}}

\nc{\BA}{{\Bbb A}} \nc{\CC}{{\Bbb C}} \nc{\DD}{{\Bbb D}}
\nc{\EE}{{\Bbb E}} \nc{\FF}{{\Bbb F}} \nc{\GG}{{\Bbb G}}
\nc{\HH}{{\Bbb H}} \nc{\LL}{{\Bbb L}} \nc{\NN}{{\Bbb N}}
\nc{\KK}{{\Bbb K}} \nc{\QQ}{{\Bbb Q}} \nc{\RR}{{\Bbb R}}
\nc{\TT}{{\Bbb T}} \nc{\VV}{{\Bbb V}} \nc{\ZZ}{{\Bbb Z}}

\nc{\cala}{{\mathcal A}} \nc{\calc}{{\mathcal C}}
\nc{\cald}{{\mathcal D}} \nc{\cale}{{\mathcal E}}
\nc{\calf}{{\mathcal F}} \nc{\calg}{{\mathcal G}}
\nc{\calh}{{\mathcal H}} \nc{\cali}{{\mathcal I}}
\nc{\call}{{\mathcal L}} \nc{\calm}{{\mathcal M}}
\nc{\caln}{{\mathcal N}} \nc{\calo}{{\mathcal O}}
\nc{\calp}{{\mathcal P}} \nc{\calr}{{\mathcal R}}
\nc{\cals}{{\mathcal S}}
\nc{\calt}{{\mathcal T}} \nc{\calw}{{\mathcal W}}
\nc{\calk}{{\mathcal K}} \nc{\calx}{{\mathcal X}}
\nc{\CA}{\mathcal{A}}

\nc{\fraka}{{\mathfrak a}}
\nc{\frakA}{{\mathfrak A}}
\nc{\frakb}{{\mathfrak b}}
\nc{\frakB}{{\mathfrak B}}
\nc{\frakC}{{\mathfrak C}}
\nc{\frakD}{{\mathfrak D}}
\nc{\frakg}{{\mathfrak g}}
\nc{\frakH}{{\mathfrak H}}
\nc{\frakL}{{\mathfrak L}}
\nc{\frakM}{{\mathfrak M}}
\nc{\bfrakM}{\overline{\frakM}}
\nc{\frakm}{{\mathfrak m}}
\nc{\frakP}{{\mathfrak P}}
\nc{\frakN}{{\mathfrak N}}
\nc{\frakp}{{\mathfrak p}}
\nc{\frakR}{{\mathfrak R}}
\nc{\frakS}{{\mathfrak S}}

\font\cyr=wncyr10
\font\cyrs=wncyr7
\nc{\li}[1]{\textcolor{red}{Li:#1}}
\nc{\fang}[1]{\textcolor{blue}{Fang: #1}}
\begin{document}

\title[Hochschild cohomology on a path algebra and Euler's formula]
{Structure of Hochschild cohomology of path
algebras and differential formulation of Euler's polyhedron
formula}
%
%=========================================================================
%
\author{Li Guo}
\address{
%School of Mathematics and Statistics, Lanzhou University, Lanzhou 730000, Gansu, China, and
Department of Mathematics and Computer Science,
         Rutgers University,
         Newark, NJ 07102}
\email{liguo@rutgers.edu}
\author{Fang Li}
\address{The corresponding author; Department of Mathematics, Zhejiang University, Hangzhou 310027, China}
\email{fangli@zju.edu.cn}

\subjclass[2000]{
16E40, %Cohomology of rings and algebras
16G20, %representations of quivers and partially ordered sets
05E15, %combinatorial aspects of groups and algebras
05C25, %graphs and abstract algebra (groups, rings, fields, etc)
12H05, %differential algebra
16W25, %derivations, action of Lie algebras
16S32 %rings of differential operators
%13N10 %rings of differential operators and their modules
}

\keywords{quiver, path algebra, Hochschild cohomology, Lie algebra, differential algebra, graph, Euler's polyhedron formula, connection matrix}

%========================================================================

%\date{\today}
%========================================================================

%\begin{document}
%========================================================================
\begin{abstract}
This article studies the Lie algebra $\diff(\bfk\Gamma)$ of derivations on the path algebra $\bfk\Gamma$ of a quiver $\Gamma$ and the Lie algebra on the first Hochschild cohomology group $HH^1(\bfk\Gamma)$. We relate these Lie algebras to the algebraic and combinatorial properties of the path algebra. Characterizations of derivations on a path algebra are obtained, leading to a canonical basis of $\diff(\bfk\Gamma)$ and its Lie algebra properties. Special derivations are associated to the vertices, arrows and faces of a quiver, and the concepts of a \conmat and \bounmat are introduced to study the relations among these derivations, concluding that the space of edge derivations is the direct sum of the spaces of the vertex derivations and the face derivations, while the dimensions of the latter spaces are the largest possible. By taking dimensions, this relation among spaces of derivations recovers Euler's polyhedron formula.
This relation also leads a combinatorial construction of a canonical basis of the Lie algebra $HH^1(\bfk\Gamma)$, together with a semidirect sum decomposition of $HH^1(\bfk\Gamma)$.
\end{abstract}

\maketitle

\tableofcontents

\setcounter{section}{0}

%========================================================================

\section{Introduction}
\mlabel{sec:intr}

This paper studies the structure of the Lie algebra of derivations on the path algebra $\bfk\Gamma$ of a quiver $\Gamma$ and the Lie algebra of outer derivations on the path algebra, also known as the first Hochschild cohomology group $HH^1(\bfk\Gamma)$. This study has two motivations, one from Hochschild cohomology and one from differential algebra. 
We determine a canonical basis and their multiplication constants for these two Lie algebras, and relate it to the combinatorial properties of the quiver, such as Euler's Polyhedron Theorem.

The study of Hochschild cohomology of quiver related algebras started with the dimension formula for $HH^n(\bfk\Gamma)$ given by Happel in 1989~\cite{Ha}, who showed that for an acyclic quiver $\Gamma$ and a field $\bf k$,
\begin{equation}
HH^0({\bf k}\Gamma)=\bfk,\quad \dim_{\bf k}HH^1({\bf k}\Gamma)=1-\mid
V\mid+\sum_{\alpha\in E}v(\alpha),\quad HH^i({\bf k}\Gamma)=0,~
\forall i\geq 2
\notag
\end{equation}
where
$v(\alpha)=\dim_{\bf k}t(\alpha){\bf k}\Gamma h(\alpha)$, $V$ and $E$
are respectively the sets of vertices and arrows of $\Gamma$.
Afterwards, there have been extensive studies on the dimensions of the Hochschild cohomology groups of quiver related algebras, such as  truncated path algebras, monomial algebras, schurian algebras and 2-nilpotent algebras~\mcite{ACT,C1,Ha,Lo,Pe,Sa,St,Xu}.

Further understanding of the Lie algebra $HH^1(\bfk\Gamma)$ would benefit from an explicit structure of this Lie algebra, such as a canonical basis and the corresponding multiplication constants. This is what we would like to achieve in this
paper. We find that the choice of the basis of
$HH^1({\bf k}\Gamma)$ is related to the combinatorics, such as Euler's formula, and the topology, such as the genus, of the quiver.

Our second motivation is differential algebra which has its origin in the algebraic study of differential equations~\mcite{Ko,Ri,SP} and is a natural yet profound extension of commutative algebra and the related algebraic geometry. After many years of developments, the theory has expanded into a vast area in mathematics~\mcite{CGKS,Ko,SP}. Furthermore, differential algebra has found important applications in arithmetic geometry, logic and computational algebra, especially in the profound work of W. Wu on mechanical proof of geometric theorems~\mcite{Wu,Wu2}.

Most of the study on differential algebra has been for commutative algebras and fields. Recently, there have been interests to study differential algebra for noncommutative algebras. For instance, in connection with combinatorics, differential structures were found on heap ordered trees~\mcite{GL} and on decorated rooted trees~\mcite{GK3}.

%This paper can be regarded as additional step on differential study of noncommutative algebras, by considering differential algebra structures on another combinatorially defined object, namely the path algebra of a quiver. This gives a natural class of differential algebras of finite and infinite dimensions.

This paper gives a differential study of the path algebra of a
quiver, as a first step in the study of differential structures on
Artinian algebras. According to the well-known Gabriel
Theorem~\mcite{ASS,ARS}, a basic algebra over an algebraically
closed field is a quotient of the path algebra of its Ext-quiver
modulo an admissible ideal. More generally, by~\mcite{LL}, an
Artinian algebra over a perfect field is isomorphic to a
quotient of the generalized path algebra of its natural quiver. Thus if we can determine the differential structures on
path algebras (resp. generalized path algebras), including their
differential ideals, then by taking the quotients of these algebras
modulo their differential ideals, we will be able to obtain the
differential structure on a basic algebra (resp. an Artinian
algebra). For more related references, see \cite{Li,Li2,LC}.

In Section~\mref{sec:prop} we characterize when a linear operator on a path algebra is a derivation. These characterizations of derivations allow us to obtain in Section~\mref{sec:str} a canonical basis of the Lie algebra of derivations on a path algebra and obtain a structure theorem of $\diff(\bfk\Gamma)$. This structure theorem is then applied to study Lie algebra properties of $\diff(\bfk\Gamma)$. In Section~\mref{sec:comb}, we focus on three types of derivations of combinatorial nature, namely derivations from the vertices, arrows and faces of the quiver respectively. Dimension formulas of the spaces spanned by these derivations are proved and the relations among them are determined. In Section~\mref{sec:app} we give two applications of these dimension formulas. We first revisit Euler's Polyhedron Theorem from a differential viewpoint, and prove that the linear space of edge derivations is the direct sum of the linear spaces of vertex derivations and face derivations. Taking dimensions of the spaces in this direct sum decomposition gives the original formula of Euler. We next apply the combinatorial derivations to obtain a canonical basis for $HH^1(\bfk\Gamma)$. This basis allows us to do computations in this Lie algebra and factor it as a semidirect sum of an abelian Lie subalgebra and an Lie ideal.

\delete{
In our study of outer differential operators \fang{ and the first
Hochschild cohomology}, it is natural to view a planar quiver as
embedded into the Riemann sphere rather than into the plane. In this
way, algebraic invariants of the space of outer differential
operators of a quiver (on the Riemann sphere), such as its rank and
canonical basis, is naturally related to the topological and graphic
invariants, such as the genus and Euler's characteristic of the
Riemann sphere. We hope to establish such a relationship for the
more general quivers that can be embedded into a Riemann surface of
higher genus.
}

\section{Derivations on path algebras}\mlabel{sec:prop}
The main purpose of this section is to provide necessary and sufficient conditions for  a linear operator on a path algebra to be a differential operator.
%, and meantime some properties of differential operators on the path algebra are given.

\subsection{Derivations}
\mlabel{ss:diffop}
We briefly recall concepts, notations and facts on differential algebras, Lie algebras and path algebras of quivers. Further details on these three subjects can be found in~\mcite{CGKS,GK3,Ko}, in~\mcite{Hu} and in~\mcite{CB,ASS,ARS,Li}, respectively.

Let $\bfk$ be a field and let $A$ be a $\bfk$-algebra.
Let $Lie(A)=(A,[,])$ denote the Lie algebra structure on $A$ with the Lie bracket
\begin{equation}
[x,y]:=xy-yx, \quad x,y \in A.
\notag%\mlabel{eq:liea}
\end{equation}
A {\bf derivation} (or a {\bf differential operator}) on $A$ is a $\bfk$-linear map $D:A\to A$ such that
\begin{equation}
D(xy)=D(x)y+xD(y), \quad \forall x,y\in A.
\notag %\mlabel{eq:diff}
\end{equation}
Let $\diff(A)$ denote the set of derivations on $A$.
Then with the Lie bracket
\begin{equation}
[D_1,D_2]:=D_1\circ D_2-D_2\circ D_1, \quad D_1,D_2\in\diff(A),
\notag %\mlabel{eq:difprod}
\end{equation}
$\diff(A)$ is a Lie algebra, called the {\bf Lie algebra of derivations} on $A$.

For $a\in A$, define the {\bf inner derivation}
\begin{equation}
D_a: A \to A, \quad D_a(b)=(ad_a)(b):=ab-ba, \quad b\in A.
\mlabel{eq:ind}
\end{equation}
Then the map
\begin{equation}
\frakD: Lie(A) \to \diff(A), \quad \frakD(a)=D_a, \quad a\in A,
\mlabel{eq:indh}
\end{equation}
from Eq.~(\mref{eq:ind}) is a Lie algebra homomorphism whose kernel
is $C(A)$, the center of $A$, and also the zeroth Hochschild cohomology group $HH^0(A)$.

The subset $\; \indiff(A):=\im \frakD \subseteq \diff(A)\;$ is a Lie ideal.
The quotient Lie algebra
\begin{equation}
\outdiff(A):=\diff(A)/\indiff(A) \mlabel{eq:outd}
\end{equation}
is called the {\bf Lie algebra of outer
derivations}. As is well-known~\cite[\S 11.5]{Pi},
$\outdiff(A)$ is also the {\bf first Hochschild cohomology group} $HH^1(A)$.

We will study $\diff(A)$ and $HH^1(A)$ when $A=\bfk \Gamma$ is the
path algebra of a connected quiver $\Gamma$.

\subsection{Path algebras}

A {\bf quiver} is a quadruple $\Gamma=(V,E,t,h)$ consisting of a set $V$ of {\bf vertices}, a set $E$ of {\bf arrows} and a pair of maps $h,t: E\to V$. When there is no danger of confusion, we also denote $\Gamma=(V,E)$.
A quiver is called {\bf trivial} if $E=\emptyset$.
Let $\calp$ denote the set of paths of $\Gamma$. For $p\in \calp$ let $t(p)$ and $h(p)$ denote the {\bf tail} and {\bf head} of $p$.

Let
\begin{equation}
\bfk\Gamma = \bigoplus_{p\in \calp} \bfk p,
\notag %\mlabel{eq:pathalg}
\end{equation}
denote the {\bf path algebra} of $\Gamma$ where the product is given by
\begin{equation}
p \cdot q:= \delta_{h(p),t(q)}pq:=\left\{\begin{array}{ll} pq, & h(p)=t(q), \\ 0, & \text{otherwise}. \end{array} \right .
\notag %\mlabel{eq:paprod}
\end{equation}
Here $\delta_{h(p),t(q)}$ is the delta function.
To simplify notations, we often suppress the symbol $\cdot$ and denote $p \cdot q = pq$, with the convention that $pq=0$ when $h(p)\neq t(q)$.

We will use the following notations on quivers and their path algebras.
\begin{defn}
{\rm
\begin{enumerate}
\item
Let $p$ be a path in $\Gamma$ consisting of the ordered list $v_0, p_1, v_1,
\cdots, v_{\ell-1}, p_\ell, v_\ell,$ with $v_i\in V, 0\leq i\leq
\ell$ and $p_j\in E, 1\leq j\leq \ell$, such that $t(p_j)=v_{j-1}$
and $h(p_j)=v_j$, $1\leq j\leq \ell$.
The integer $\ell\geq 0$ is called the {\bf length} of the path $p$ and is denoted by $\ell(p)$.
\item
The expression
\begin{equation}
p=v_0 p_1 v_1 \cdots v_{\ell-1} p_\ell v_\ell,
\notag %\mlabel{eq:st}
\end{equation}
is called the {\bf standard decomposition of $p$}, and the expression
\begin{equation}
p=p_1\cdots p_\ell,
\notag %\mlabel{eq:ar}
\end{equation}
is called the {\bf decomposition of $p$ into arrows}.
Both decompositions are unique.
\item
For two paths $p$ and $q$, denote $p\parallel q$ and called $p$ and $q$ {\bf parallel}, if $t(p)=t(q)$ and $h(p)=h(q)$.
\item
For two paths $p$ and $q$, if $q=pr$ (resp. $q=rp$) for some path $r$, then call $p$ a {\bf tail} (resp. {\bf head}) of $q$ and denote by $p\mtail q$ (resp. $p\mhead q$). Such an $r$ is unique for given $p,q$.
\item
A path $p$ is called {\bf acyclic} if $h(p)\neq t(p)$.
The set of acyclic paths is denoted by $\calpa$.
\item
A quiver $\Gamma$ is called {\bf acyclic} if $\Gamma$ has no oriented cycles, that is, $\calp\backslash V =\calpa$.
\end{enumerate}
}
\end{defn}

In this paper, we always assume that a quiver
$\Gamma$ is finite, that is, its vertex set and arrow set are both finite.

\begin{lemma}
Let $A$ be a $\bfk$-algebra with a linear basis $X$. Then a linear operator $D:A \to A$ is a derivation if and only if
\begin{equation}
 D(xy) = D(x)y + x D(y), \quad \forall x, y\in X.
 \mlabel{eq:base}
\end{equation}
In particular, a linear operator $D:\bfk\Gamma \to \bfk\Gamma$ is a derivation if and only if Eq.~(\mref{eq:base}) holds for all $x,y\in \calp$.
\mlabel{lem:base}
\end{lemma}
\begin{proof}
The ``only if" part is clear. Conversely, suppose the condition
holds. Let $u,v$ be in $A$. Then $ u = \sum_{x\in X} c_x x,\; v=
\sum_{y\in X} d_y y.\;$ Thus, we have $$D(uv)=  \sum_{x,y\in X} c_x
d_y D(xy)= \sum_{x,y\in X} c_x d_y (D(x)y+xD(y)) = D(u)v + uD(v).$$
This is what we need.
\end{proof}

\subsection{Necessary and sufficient conditions for a derivation}
We now characterize a derivation on a path algebra $\bfk \Gamma$ in terms of the paths $\calp$ of $\Gamma$. These characterizations will be applied in the next section to determine all derivations on a path algebra.

Let $D:\bfk\Gamma \to \bfk\Gamma$ be a linear operator. Then for any $p\in \calp$,
\begin{equation}
D(p)= \sum_{q\in \calp} c^p_q q,
\notag %\mlabel{eq:dcoef}
\end{equation}
for unique $c^p_q\in \bfk$. We will use this notation for the rest of this paper. We also use the convention that, for the empty set $\emptyset$,
\begin{equation}
\sum\limits_{q\in \emptyset} c^p_q q= 0.
\notag %\mlabel{eq:empty}
\end{equation}

\begin{theorem}
Let $\Gamma$ be a quiver.
A linear operator $D: \bfk\Gamma \to \bfk\Gamma$ is a derivation if and only if $D$ satisfies the following conditions.
\begin{enumerate}
\item
For $v\in V$,
\begin{equation}
D(v)=
 \sum\limits_{q\in \calpa, t(q)=v} c^v_q q
+\sum\limits_{q\in \calpa, h(q)=v} c^v_q q
= \sum\limits_{q\in \calpa, t(q)=v \text{ or } h(q)= v} c^v_q q.
\mlabel{eq:vdiff}
\end{equation}
\item
For $p\in \calp\backslash V$,
\begin{equation}
D(p)=\sum\limits_{q\in \calpa,\; h(q)=t(p)} c^{t(p)}_q qp +
\sum\limits_{q\parallel p} c^p_q q +
\sum\limits_{q\in \calpa,\; t(q)=h(p)} c^{h(p)}_q pq,
\mlabel{eq:pdiff}
\end{equation}
where the coefficients $c^p_q$ are subject to the following conditions.
\begin{enumerate}
\item
For any path $q\in \calpa$,
\begin{equation}
c^{h(q)}_q + c^{t(q)}_q = 0.
\mlabel{eq:pathrel}
\end{equation}
\item
For any path $p=p_1p_2$ with $p_1,p_2\in \calp\backslash V$ and
$q\parallel p$, we have
\begin{equation}
c^p_q=c^{p_1p_2}_q = \left \{ \begin{array}{ll}
 c^{p_1}_{q_1}+c^{p_2}_{q_2}, & \text{if\;\;\;} p_2\mhead q \text{ with } q=q_1p_2 \text{ and } p_1\mtail q \text{ with } q=p_1q_2, \\
 c^{p_1}_{q_1}, & \text{if\;\;\;} p_2\mhead q \text{ with } q=q_1p_2 \text{ and } q_1\not\mtail p_1, \\
 c^{p_2}_{q_2}, & \text{if\;\;\;} p_1\mtail q \text{ with } q=p_1q_2 \text{ and } q_2\not\mhead p_2,\\
 0, & \text{if\;\;\;} p_2\not\mhead q \text{ and } p_1\not\mtail q.
 \end{array} \right.
 \mlabel{eq:prodrel}
 \end{equation}
\end{enumerate}
\end{enumerate}
\mlabel{thm:dcond}
\end{theorem}

\begin{proof}
($\Longrightarrow$) Let $D:\bfk\Gamma\to \bfk\Gamma$ be a linear operator. For a given $v\in V$, since $vv =v$, we have $
D(v)=D(vv)=D(v)v+vD(v).$ Thus
\begin{equation}
 \sum_{q\in \calp} c^v_q q = \Big(\sum_{q\in \calp}c^v_q q\Big)v + v \Big(\sum_{q\in \calp} c^v_q q\Big)
= \sum_{q\in \calp, h(q)=v}c^v_q q + \sum_{q\in \calp, t(q)=v}
c^v_q\, q \mlabel{eq:idem1}
\end{equation}
since $qv=0$ unless $h(q)=v$ and $vq=0$ unless $t(q)=v$.

Similarly, from
$$D(v)=D(v^3) = D(v)v^2+vD(v)v + v^2D(v)=D(v)v+vD(v)v+vD(v),$$
we have
\begin{eqnarray}
 \sum_{q\in \calp} c^v_q q &=& \Big(\sum_{q\in \calp}c^v_q q\Big)v + v\Big(\sum_{q\in \calp}c^v_q q\Big)v+ v \Big(\sum_{q\in \calp} c^v_q q\Big)
 \notag\\
&=& \sum_{q\in \calp, h(q)=v}c^v_q q + \sum_{q\in \calp,
h(q)=v,t(q)=v}c^v_q q + \sum_{q\in \calp, t(q)=v} c^v_q\, q.
\notag %\mlabel{eq:item2}
\end{eqnarray}
Comparing this with Eq.~(\mref{eq:idem1}), we obtain $
\sum\limits_{q\in \calp, h(q)=v,t(q)=v}c^v_q q=0.$ Then
Eq.~(\mref{eq:vdiff}) follows from Eq.~(\mref{eq:idem1}).

\medskip

Also, for a given path $p\in \calp\backslash V$, we have
\begin{eqnarray*}
D(p)&=&D(t(p)ph(p))\\
&=& D(t(p))ph(p)+t(p)D(p)h(p)+t(p)p D(h(p))\\
&=& D(t(p))p+t(p)D(p)h(p)+p D(h(p))\\
&=& D(t(p))p+\sum_{q\parallel p} c^p_q\,q+p D(h(p)).
\end{eqnarray*}
By Eq.~(\mref{eq:vdiff}), we have
$$
D(t(p))p= \Big(\sum\limits_{q\in \calpa, t(q)=t(p)} c^{t(p)}_q q
+\sum\limits_{q\in \calpa, h(q)=t(p)} c^{t(p)}_q q\Big)p =
\sum\limits_{q\in \calpa, h(q)=t(p)} c^{t(p)}_q qp$$ since $qp=0$ if
$h(q)\neq t(p)$. Similarly, $pD(h(p))= \sum\limits_{q\in \calpa,
t(q)=h(p)} c^{h(p)}_q pq.$ This proves Eq.~(\mref{eq:pdiff}).
\smallskip

Thus we only need to prove Eq.~(\mref{eq:pathrel}) and Eq.~(\mref{eq:prodrel}) in order to complete the proof of ($\Longrightarrow$). For this purpose, we prove a lemma.

\begin{lemma}
\begin{enumerate}
\item
Let $p_1,p_2\in \calp$. Suppose $p_1p_2=0$ and Eq.~(\mref{eq:vdiff}) and Eq.~(\mref{eq:prodrel}) hold for $p_1$ and $p_2$. Then $D(p_1p_2)=D(p_1)p_2+p_1D(p_2)$ if and only if, for every $q\in \calp$ with $t(q)=p_1$ and $h(q)=p_2$, Eq.~(\mref{eq:pathrel}) holds.
\mlabel{it:leib0}
\item
Let $p_1, p_2\in \calp\backslash V$. Suppose
$p_1p_2\neq 0$ and Eq.~(\mref{eq:pdiff}) holds for $p_1, p_2$ and $p_1p_2$. Then $D(p_1p_2)=D(p_1)p_2+p_1D(p_2)$ if and only
if, for every $q\parallel p$, it
holds that Eq.~(\mref{eq:prodrel}). \mlabel{it:leib1}
\end{enumerate}
\mlabel{lem:dzero}
\end{lemma}

\begin{proof}
(\mref{it:leib0}). Since $p_1p_2=0$ and $D$ is linear, we have $D(p_1p_2)=0$. To compute $D(p_1)p_2+p_1D(p_2)$, first consider the case when $p_1, p_2$ are in $V$.
{\allowdisplaybreaks
\begin{eqnarray*}
&& D(p_1)p_2 + p_1D(p_2)\\
&=& \left(\sum_{q\in \calpa, h(q)=p_1}c_q^{p_1} q +\sum_{q\in \calpa, t(q)=p_1} c^{p_1}_q\,q \right)p_2
\\
&&+ p_1\left( \sum_{q\in\calpa, h(q)=p_2}c^{p_2}_q\,q +\sum_{q\in\calpa, t(q)=p_2}c^{p_2}_q\,q\right)\qquad \text{(by Eq.~(\mref{eq:pdiff}))}\\
&=&
\sum_{q\in\calpa,t(q)=p_1,h(q)=p_2} c^{p_1}_q q
+\sum_{q\in\calpa, t(q)=p_1,h(q)=p_2} c^{p_2}_q q \qquad \text{(since }p_1\neq p_2 \text{)}\\
&=& \sum_{q\in\calpa,t(q)=p_1,h(q)=p_2} (c^{p_1}_q+c^{p_2}_q ) q.
\end{eqnarray*}
}
Thus $D(p_1)p_2
+ p_1D(p_2)=0$ if and only if all the coefficients in the last sum
are zero. That is, $c^{h(q)}_q+c^{t(q)}_q=0$ for all $q\in
\calp$ with  $t(q)=p_1,h(q)=p_2$. This proves
(\mref{it:leib0}).

Next consider the case when $p_1$ and $p_2$ are in $\calp\backslash V$. Then we have
{\allowdisplaybreaks
\begin{eqnarray*}
&& D(p_1)p_2 + p_1D(p_2)\\
&=& \left(\sum_{q\in \calpa, h(q)=t(p_1)}c_q^{t(p_1)} qp_1
+\sum_{q\parallel p_1} c^{p_1}_q\,q +\sum_{q\in\calpa,t(q)=h(t_1)}c^{h(p_1)}_q\,p_1q \right) p_2 \\
&& + p_1\left(\sum_{q\in \calpa, h(q)=t(p_2)}c_q^{t(p_2)} qp_2
+\sum_{q\parallel p_2} c^{p_2}_q\,q +\sum_{q\in\calpa,t(q)=h(t_2)}c^{h(p_2)}_q\,p_2q \right) \qquad \text{(by Eq.~(\mref{eq:pdiff}))}\\
&=&
\sum_{t(q)=h(p_1),h(q)= t(p_2)} c^{h(p_1)}_q p_1q p_2
+ \sum_{h(q)=t(p_2),t(q)= h(p_1)} c^{t(p_2)}_qp_1q p_2 \qquad \text{(since } h(p_1)\neq t(p_2) \text{)}\\
&=& \sum_{t(q)=h(p_1),h(q)=t(p_2)} (c^{h(p_1)}_q+c^{t(p_2)}_q )
p_1qp_2.
\end{eqnarray*}
}
For any two distinct $q\in \calp$ with $t(q)=h(p_1)$ and
$h(q)=t(p_2)$, the corresponding $p_1qp_2$ are non-zero and are
distinct paths in the basis $\calp$ of $\bfk\Gamma$. Thus $D(p_1)p_2
+ p_1D(p_2)=0$ if and only if all the coefficients in the last sum
are zero. That is, $c^{h(p_1)}_q+c^{t(p_2)}_q=0$ for all $q\in
\calp$ with  $t(q)=h(p_1),h(q)=t(p_2)$. This proves (\mref{it:leib0}) in this case.

The cases when one of $p_1,p_2$ is in $V$ and the other one is in $\calp\backslash V$ can be verified in the same way.

\medskip

\noindent
(\mref{it:leib1}).
By Eq.~(\mref{eq:pdiff}) we have
\begin{equation}
D(p_1p_2)= \sum_{q\in\calpa,h(q)=t(p_1)}c^{t(p_1)}_q\,qp_1p_2 + \sum_{q\parallel p_1p_2} c^{p_1p_2}_q q + \sum_{q\in\cala,t(q)=h(p_2)}c^{h(p_2)}_p\,p_1p_2q
\mlabel{eq:prodrel1}
\end{equation}
since $t(p_1p_2)=t(p_1)$ and $h(p_1p_2)=h(p_2)$.
Similarly,
\begin{eqnarray*}
D(p_1)p_2 &=& \big( \sum_{q\in\calpa,h(q)=t(p_1)}c^{t(p_1)}_q\,qp_1 + \sum_{q_1\parallel p_1} c^{p_1}_{q_1} q_1 + \sum_{q\in\calpa,t(q)=h(p_1)}c^{h(p_1)}_q\,p_1q \big) p_2.
\end{eqnarray*}
Since $h(q)\neq t(q)=h(p_1)=t(p_2)$, we have $qp_2=0$ for $q$ in the last sum. Thus we obtain \begin{equation}
D(p_1)p_2 = \sum_{q\in\calpa,h(q)=t(p_1)}c^{t(p_1)}_q\,qp_1p_2 + \sum_{q_1\parallel p_1} c^{p_1}_{q_1} q_1p_2.
\mlabel{eq:prodrel2}
\end{equation}
By the same argument, we have
\begin{equation}
p_1D(p_2)= \sum_{q_2\parallel p_2} c^{p_2}_{q_2} p_1 q_2 + \sum_{q\in\calpa,t(q)=h(p_2)}c^{h(p_2)}_q\,p_1p_2q.
\mlabel{eq:prodrel3}
\end{equation}
Thus by equations (\mref{eq:prodrel1}), (\mref{eq:prodrel2}) and (\mref{eq:prodrel3}), we see that $D(p_1p_2)=D(p_1)p_2 + p_1 D(p_2)$ if and only if
\begin{equation}
\sum_{q\parallel p_1p_2} c^{p_1p_2}_q q =
\sum_{q_1\parallel p_1} c^{p_1}_{q_1} q_1
p_2 + \sum_{q_2\parallel p_2}
c^{p_2}_{q_2} p_1 q_2. \mlabel{eq:prodrel4}
\end{equation}

In the sum on the left hand side, the paths
$q\parallel p_1p_2$ can be divided into the disjoint union of the following four subsets: \begin{eqnarray}
\calp_1&:=&\{q\in \calp\ |\ p\parallel p_1p_2, p_1\mtail q \text{ and } p_2\mhead q\},\\
\calp_2&:=&\{q\in \calp\ |\ q\parallel p_1p_2, p_1\mtail q \text{ and } p_2\not\mhead q\},\\
\calp_3&:=&\{q\in \calp\ |\ q\parallel p_1p_2, p_1\not\mtail q \text{ and } p_2\mhead q\},\\
\calp_4&:=&\{q\in \calp\ |\ q\parallel p_1p_2, p_1\not\mtail q \text{ and } p_2\not\mhead q\}.
\end{eqnarray}
Thus the left hand side of Eq.~(\mref{eq:prodrel4}) becomes
\begin{equation}
\sum_{q\parallel p_1p_2} c^{p_1p_2}_q q =
\sum_{q\in \calp_1} c^{p_1p_2}_q q +
\sum_{q\in \calp_2} c^{p_1p_2}_q q +
\sum_{q\in \calp_3} c^{p_1p_2}_q q +
\sum_{q\in \calp_4} c^{p_1p_2}_q q.
\mlabel{eq:lhs}
\end{equation}
By the definitions of $\mtail$ and $\mhead$, for the two sums on the right hand side of Eq.~(\mref{eq:prodrel4}), we have  respectively
\begin{equation}
\{q_1p_2\ |\ q_1\parallel p_1\} = \calp_1 \cup \calp_3
\mlabel{eq:rhs1}
\end{equation}
and
\begin{equation}
\{p_1q_2\ |\ q_2\parallel p_2\} = \calp_1 \cup \calp_2.
\notag %\mlabel{eq:rhs2}
\end{equation}
Thus the right hand side of Eq.~(\mref{eq:prodrel4}) becomes \begin{equation}
\sum_{q\in \calp_1 \text{ with } q=q_1p_2=p_1q_2}(c^{p_1}_{q_1}+c^{p_2}_{q_2})q
+ \sum_{q\in\calp_3 \text{ with } q=q_1p_2} c^{p_1}_{q_1} q_1
p_2 + \sum_{q\in \calp_2 \text{ with } q=p_1q_2} c^{p_2}_{q_2}
p_1 q_2. \mlabel{eq:prodrel4-2}
\end{equation}
Now comparing the coefficients on the two sides of
Eq.~(\mref{eq:prodrel4}) using Eq.~(\mref{eq:lhs}) and
Eq.~(\mref{eq:prodrel4-2}), we obtain Eq.~(\mref{eq:prodrel}).
\end{proof}

Now we return to the proof of Theorem~\mref{thm:dcond}. For any path $p\in \calp\backslash V$ with $t(p)\neq h(p)$, we have $t(p) h(p) = 0$. So applying Lemma~\mref{lem:dzero}.(\mref{it:leib0}) to $p_1=t(p)$ and $p_2=h(p)$ in Eq.~(\mref{eq:vdiff}), we obtain Eq.~(\mref{eq:pathrel}).

Finally let $p=p_1p_2$ with $p_1,p_2\in \calp\backslash V$. Applying Lemma~\mref{lem:dzero}.(\mref{it:leib1}) to $p=p_1 p_2$, we obtain Eq.~(\mref{eq:prodrel}).
\smallskip

\noindent ($\Longleftarrow$)  Suppose a linear operator
$D:\bfk\Gamma \to \bfk\Gamma$ is given by Eq.~(\mref{eq:vdiff}) and
(\mref{eq:pdiff}) subject to the conditions Eq.~(\mref{eq:pathrel})
and (\mref{eq:prodrel}). By Lemma~\mref{lem:base}, to show that $D$
is a derivation we just need to show that
Eq.~(\mref{eq:base}) holds for $X= \calp$. Thus we only need to
verify Eq.~(\mref{eq:base}) in the following four cases.
\begin{enumerate}
\item
$x,y\in V$;
\mlabel{it:vv}
\item
$x\in V, y\in \calp\backslash V$;
\mlabel{it:vp}
\item
$x\in \calp\backslash V, y\in V$;
\mlabel{it:pv}
\item
$x,y\in \calp\backslash V.$
\mlabel{it:pp}
\end{enumerate}

\noindent
{\bf Case~(\mref{it:vv}). } Let $x,y\in V$. If $x=y$, then $xy=x$. So by Eq.~(\mref{eq:vdiff}),
\begin{equation}
D(xy)=D(x)= \sum_{q\in \calpa, t(q)=x \text{ or } h(q)=x} c^x_q q.
\notag %\mlabel{eq:vv1}
\end{equation}
Also by Eq.~(\mref{eq:vdiff}), we have
\begin{eqnarray*}
D(x)x + x D(x) &=& \Big(\sum_{q\in \calpa, t(q)=x \text{ or } h(q)=x} c^x_q q\Big) x
+ x \Big( \sum_{q\in \calpa, t(q)=x \text{ or } h(q)=x} c^x_q q \Big) \\
&=& \sum_{q\in \calpa, h(q)=x} c^x_q q
+ \sum_{q\in \calpa, t(q)=x} c^x_q q.
\end{eqnarray*}
This verifies Eq.~(\mref{eq:base}).

If $x\neq y$, then $xy=0$. By Eq.~(\mref{eq:vdiff}), we have
\begin{eqnarray*}
&& D(x)y+xD(y)\\
&=& \big (\sum_{q\in \calpa, t(q)=x} c^x_q q + \sum\limits_{q\in
\calpa, h(q) = x} c^x_q q\big) y + x \big (\sum_{q\in \calp_A,
t(q)=y} c^y_q q + \sum\limits_{q\in \calpa, h(q) = y} c^y_q q\big)
\\
&=& \sum\limits_{q\in \calp_A, t(q)=x, h(q)= y} c^x_q q +
\sum\limits_{q\in \calp_A, t(q)= x, h(q) = y} c^y_q q
\\
&=& \sum\limits_{q\in \calp_A, t(q)=x, h(q)= y} (c^x_q + c^y_q) q,
\end{eqnarray*}
which is zero by Eq.~(\mref{eq:pathrel}). This again verifies Eq.~(\mref{eq:base}) in this case.
\medskip

\noindent
{\bf Case~(\mref{it:vp}). } Let $x\in V$ and $y\in \calp\backslash V$. If $t(y)=x$, then $xy=y$. By Eq.~(\mref{eq:pdiff}) we have
\begin{equation}
D(y)= \sum\limits_{q\in \calpa, h(q)=t(y)} c^{t(y)}_q q y +
\sum\limits_{q\parallel y} c^y_q q +
\sum\limits_{q\in \calpa, t(q)=h(y)} c^{h(y)}_q y q. \mlabel{eq:vp1}
\end{equation}
Thus by Eqs.~(\mref{eq:vdiff}) and (\mref{eq:pdiff}) we have
{\allowdisplaybreaks
\begin{eqnarray*}
&&D(x)y+xD(y)\\
&=& \big(\sum\limits_{q\in \calpa, t(q)=x \text{ or } h(q)=x} c^x_q q \big) y  +
x \big( \sum\limits_{q\in \calpa, h(q)=t(y)} c^{t(y)}_q q y + \sum\limits_{q\parallel y} c^y_q q + \sum\limits_{q\in \calpa, t(q)=h(y)} c^{h(y)}_q y q \big)\\
&=&
\sum\limits_{q\in \calpa, h(q)=x} c^x_q q y + \sum\limits_{q\in \calpa, h(q)=t(y), t(q)=x} c^{t(y)}_q q y +
\sum\limits_{q\parallel y} c^y_q q +
\sum\limits_{q\in \calpa, t(q)=h(y)} c^{h(y)}_q y q.
\end{eqnarray*}
}
since $qy=0$ unless $h(q)=t(y)=x$ and $xq=0$ unless $t(q)=t(y)=x$ in which case $xq=q$.
Since $t(y)=x$, the right hand side agrees with the right hand side of Eq.~(\mref{eq:vp1}), proving Eq.~(\mref{eq:base}).

If $t(y)\neq x$, then $xy=0$. By the same argument as in Case~(\mref{it:vv}), we obtain
$$
D(x)y+xD(y) =
\sum\limits_{q\in \calp, t(q)=x, h(q)=t(y)}(c^{t(q)}_q+c^{h(q)}_q) qy$$
which is zero by Eq.~(\mref{eq:pathrel}). Thus Eq.~(\mref{eq:base}) holds.

\medskip

\noindent
{\bf Case~(\mref{it:pv}). } The proof in this case is the same as that of Case~(\mref{it:vp}).

\medskip

\noindent
{\bf Case~(\mref{it:pp}). } Let $x, y$ be in $\calp\backslash V$. If $h(x)=t(y)$. Then $p:=xy$ is a path. So the conclusion follows by taking $p_1=x$, $p_2=y$ in Lemma~\mref{lem:dzero}.(\mref{it:leib1}).

If $h(x)\neq t(y)$, then $xy=0$ and then Eq.~(\mref{eq:base}) follows from Lemma~\mref{lem:dzero}.(\mref{it:leib0}).
\end{proof}

\subsection{A variation of Theorem~\mref{thm:dcond}}

For the convenience of later applications, we give another formulation of Theorem~\mref{thm:dcond} on the condition of a derivation on a path algebra.

\begin{coro}
Let $\Gamma$ be a quiver.
A linear operator $D: \bfk\Gamma \to \bfk\Gamma$ is a derivation if and only if $D$ is determined by its action on the basis $\calp$ as follows.
\begin{enumerate}
\item
Let $v\in V$. Then
\begin{equation}
D(v)=
 \sum_{q\in \calpa, t(q)=v} c^{t(q)}_q q
-\sum_{q\in \calpa, h(q)=v} c^{t(q)}_q q. \mlabel{eq:vdiff2}
\end{equation}
\item
Let $p\in \calp\backslash V$. Then
\begin{equation}
D(p)= \sum_{q\in \calpa, t(q)=h(p)}
c^{t(q)}_q pq+ \sum_{q\parallel p} c^p_q q
- \sum_{q\in \calpa, h(q)=t(p)}
 c^{t(q)}_q qp,
\mlabel{eq:pdiff2}
\end{equation}
where the coefficients $c^p_q$ are subject to the following condition:
For any path $p=p_1p_2$ with $p_1,p_2\in \calp\backslash V$ and
$q\parallel p$, we have
\begin{equation}
c^p_q=c^{p_1p_2}_q = \left \{ \begin{array}{ll}
 c^{p_1}_{q_1}+c^{p_2}_{q_2}, & \text{if\;\;\;} p_2\mhead q \text{ with } q=q_1p_2 \text{ and } p_1\mtail q \text{ with } q=p_1q_2, \\
 c^{p_1}_{q_1}, & \text{if\;\;\;} p_2\mhead q \text{ with } q=q_1p_2 \text{ and } q_1\not\mtail p_1, \\
 c^{p_2}_{q_2}, & \text{if\;\;\;} p_1\mtail q \text{ with } q=p_1q_2 \text{ and } q_2\not\mhead p_2,\\
 0, & \text{if\;\;\;} p_2\not\mhead q \text{ and } p_1\not\mtail q.
 \end{array} \right.
\notag % \mlabel{eq:prodrel5}
 \end{equation}
\end{enumerate}
\mlabel{co:dcond2}
\end{coro}

\begin{proof}
We only need to show that the condition Eq~(\mref{eq:pathrel}) imposed to Eq.~(\mref{eq:vdiff}) and Eq.~(\mref{eq:pdiff}) in Theorem~\mref{thm:dcond} amount to Eq.~(\mref{eq:vdiff2}) and Eq.~(\mref{eq:pdiff2}).
First, applying Eq.~(\mref{eq:pathrel}), that is $c^{h(p)}_p=-c^{t(p)}_p$ for $p\in \calp\backslash V$, to Eq.~(\mref{eq:vdiff}) gives us Eq.~(\mref{eq:vdiff2}).

Similarly apply Eq.~(\mref{eq:pathrel}) to Eq.~(\mref{eq:pdiff}). For the coefficients in the first sum, we have $c^{h(p)}_q=c^{t(q)}_q$ by the restriction of the sum. For the coefficients in the third sum of Eq.~(\mref{eq:pdiff}) we have $c^{t(p)}_q=c^{h(q)}_p=-c^{t(q)}_q$. Thus the first and third sums in Eq.~(\mref{eq:pdiff}) agree with the corresponding sums in Eq.~(\mref{eq:pdiff2}). This is what we need. \end{proof}

\section{Structure of the Lie algebra $\diff(\bfk\Gamma)$}
\mlabel{sec:str}
In this section, we apply the characterizations (Theorem~\mref{thm:dcond} and Corollary~\mref{co:dcond2}) of a derivation on a path algebra to study the Lie algebra $\diff(\bfk\Gamma)$. We first display a canonical basis for this Lie algebra and then use the basis to establish the multiplication structure of this Lie algebra. As applications, basic properties of this Lie algebra are studies. 

\subsection{The derivation $D_{r,s}$}
Let $r\in E$ and $s\parallel r$. We
construct a linear operator
\begin{equation}
D_{r,s}: \bfk \Gamma \to \bfk \Gamma
\notag %\mlabel{eq:drs}
\end{equation}
by defining $D_{r,s}(p)$ for $p\in \calp$ by induction on the length $\ell(p)$ of $p$.

When $\ell(p)=0$, i.e., when $p\in V$, we define \begin{equation}
D_{r,s}(p)=0.
 \mlabel{eq:rsrec0}
\end{equation}
Assume that $D_{r,s}(p)$ have been defined for $p\in \calp$ with $\ell(p)=n\geq 0$. Consider $p\in \calp$
with $\ell(p)=n+1$. Then $p=p_1 \tilde{p}$ with $p_1\in E$ and $\tilde{p}\in \calp$ with $\ell(\tilde{p})=n$.
 We then define
 \begin{equation}
D_{r,s}(p)=\left \{\begin{array}{ll} s \tilde{p}+p_1D_{r,s}(\tilde{p}), & p_1=r, \\
p_1D_{r,s}(\tilde{p}), & p_1\neq r.
\end{array}
\right .
\mlabel{eq:rsrec}
\end{equation}

\begin{prop} For a quiver $\Gamma=(V,E)$, let $r\in E$ and $s\parallel r$. The linear operator $D_{r,s}$ recursively defined by Eqs.~(\mref{eq:rsrec0}) and (\mref{eq:rsrec}) have the following explicit formula. For any $p\in \calp$ with the standard decomposition
 $p=v_0p_1v_1\cdots p_kv_k$ with $v_0,\cdots,v_k\in V,
p_1,\cdots,p_k\in E$, we have \begin{equation}
D_{r,s}(p)=\left\{\begin{array}{ll}
0,  & k=0, \\
\sum\limits_{i=1}^k v_0p_{i,1}v_1\cdots p_{i,k}v_k, & k>0,
\end{array} \right .
\mlabel{eq:existq}
\end{equation}
where $p_{i,j}$ or, more precisely, $p^{(r,s)}_{i,j}$, is defined by
\begin{equation}
p_{i,j}:=p^{(r,s)}_{i,j}:=\left \{\begin{array}{ll} s, & i=j, p_i=r, \\
    0, & i=j, p_i\neq r, \\ p_j, & i\neq j.
    \end{array} \right .
\mlabel{eq:pij}
\end{equation}
\mlabel{lem:drs}
\end{prop}
For example, for $p\in\calp$ with standard decomposition
$p=v_0p_1v_1p_2v_2p_3v_3p_4v_4$ where $p_1, p_3=r$ and $p_2, p_4\neq
r$,
 we have
$$D_{r,s}(p)=v_0 s v_1p_2v_2 rv_3p_4v_4 + v_0 rv_1p_2v_2 sv_3p_4v_4
$$

\begin{proof} Let $D$ be defined by Eq.~(\mref{eq:existq}). We just need to show that $D_{r,s}(p)=D(p)$ for all $p\in \calp$. We prove this by induction on $\ell(p)$.
When $\ell(p)=0$, then $D_{r,s}(p)=0=D(p)$ by the definitions of $D_{r,s}$ and $D$. Assume the equation holds for $\ell(p)=k$ for $k\geq 0$, and consider $p\in \calp$ with $\ell(p)=k+1$. Then $p=v_0p_1 \tilde{p}=p_1\tilde{p}$ with $p_1\in E$ and $\tilde{p}\in
\calp$ with $\ell(\tilde{p})=k$. Let $\tilde{p}=v_1p_2v_2\cdots p_{k+1}v_{k+1}$
be the standard decomposition of $\tilde p$. Then by the induction hypothesis we have
{\allowdisplaybreaks
\begin{eqnarray*}
D_{r,s}(p)&=&\left \{\begin{array}{ll} s \tilde{p}+p_1D_{r,s}(\tilde{p}), & p_1=r \\
p_1D_{r,s}(\tilde{p}), & p_1\neq r
\end{array} \right . \\
&=& \left \{\begin{array}{ll} v_0sv_1p_2v_2\cdots
p_{k+1}v_{k+1}+v_0p_1\sum_{i=2}^{k+1} v_1p_{i,2}v_2\cdots
p_{i,k+1}v_{k+1}
, & p_1=r \\
v_0p_1\sum_{i=2}^{k+1} v_1p_{i,2}v_2\cdots p_{i,k+1}v_{k+1}, &
p_1\neq r
\end{array}
\right . \\&=& \left \{\begin{array}{ll} v_0sv_1p_2v_2\cdots
p_{k+1}v_{k+1}+\sum_{i=2}^{k+1}v_0p_1 v_1p_{i,2}v_2\cdots
p_{i,k+1}v_{k+1}
, & p_1=r \\
\sum_{i=2}^{k+1} v_0p_1v_1p_{i,2}v_2\cdots p_{i,k+1}v_{k+1}, &
p_1\neq r
\end{array}
\right . \\
& =&\sum_{i=1}^{k+1} v_0p_{i,1}v_1\cdots p_{i,k+1}v_{k+1},
\notag % \mlabel{eq:rsrec2}
\end{eqnarray*}
}
where $p_{i,j}$ is defined by Eq.~(\mref{eq:pij}).
%$p_{i,j}:=\left \{\begin{array}{ll} s, & i=j, p_i=r, \\    0, & i=j, p_i\neq r, \\ p_j, & i\neq j.     \end{array} \right .$
Since this agrees with $D(p)$, the induction is completed.
\end{proof}

\begin{theorem}
For $r\in E$ and $s\parallel r$, the linear operator $D_{r,s}: \bfk\Gamma \to \bfk\Gamma$ defined by Eq.~(\mref{eq:rsrec0}) and Eq.~(\mref{eq:rsrec}) is a derivation.
\mlabel{thm:rsdiff}
\end{theorem}

\begin{proof}
By Lemma~\mref{lem:base}, we only need to verify
\begin{equation}
D_{r,s}(pq)=D_{r,s}(p)q + p D_{r,s}(q), \quad \forall p, q\in\calp.
\mlabel{eq:rsbase}
\end{equation}
We will prove this by induction on $\ell(p)$.

Let $\ell(p)=0$. Then $pq=q$ if $p=t(q)$ and $pq=0$ if $p\neq t(q)$. First consider the case when $D_{r,s}(q)=0$. Then both sides of Eq.~(\mref{eq:rsbase}) are zero. So we are done. Next consider the case when $D_{r,s}(q)\neq 0$. Then we have $t(D_{r,s}(q))=t(q)$ by the definition of $D_{r,s}(q)$. Thus if $t(q)\neq p$, then both sides of Eq.~(\mref{eq:rsbase}) are zero and we are done again. If $t(q)=p$, then both sides of Eq.~(\mref{eq:rsbase}) equal to $D_{r,s}(q)$, as needed.

Next assume that Eq.~(\mref{eq:rsbase}) has been proved for $p\in \calp$ with $\ell(p)=n\geq 0$ and consider $p\in \calp$ with $\ell(p)=n+1$. Then we can write $p=p_1\tilde{p}$ and obtain
{\allowdisplaybreaks
\begin{eqnarray*}
D_{r,s}(pq)&=& D_{r,s}(p_1\tilde{p}q) \\
&=& \left \{ \begin{array}{ll}
s\tilde{p}q + p_1 D_{r,s}(\tilde{p}q), & p_1=r, \\
p_1 D_{r,s}(\tilde{p}q), & p_1\neq r. \end{array}\right .
\quad \text{(by Eq.~(\mref{eq:rsrec}))} \\
&=& \left \{ \begin{array}{ll}
s\tilde{p}q + p_1 (D_{r,s}(\tilde{p})q+\tilde{p}D_{r,s}(q)), & p_1=r, \\
p_1 (D_{r,s}(\tilde{p})q+\tilde{p}D_{r,s}(q)), & p_1\neq r. \end{array}\right . \quad \text{(by induction hypothesis)}\\
&=& \left \{\begin{array}{ll}
D_{r,s}(p) q + p D_{r,s}(q), & p_1 =r, \\
D_{r,s}(p) q + p D_{r,s}(q), & p_1 \neq r. \\
\end{array} \right . \quad \text{(by Eq.~(\mref{eq:rsrec}))}
\end{eqnarray*}
}
This completes the induction.
\end{proof}

As an immediate consequence of Theorem~\mref{thm:dcond} and Theorem~\mref{thm:rsdiff}, we prove the following existence theorem of derivations on path algebras. Note that the zero map on any algebra is a derivation.

\begin{coro}
There is a nonzero derivation on the path algebra $\bfk\Gamma$ of a quiver $\Gamma$ if and only if $\Gamma$ is a non-trivial quiver, that is, $\Gamma$ has at least one arrow.
Equivalently, $\diff(\bfk\Gamma)$ is a non-zero Lie algebra if and only if $\Gamma$ is a non-trivial quiver.
\mlabel{co:exist}
\end{coro}

\begin{proof}
Suppose $\Gamma$ contains only vertices. Let $D:\bfk\Gamma \to
\bfk\Gamma$ be a derivation. Then by
Eq.~(\mref{eq:vdiff}), we have $ D(v)=
  \sum\limits_{q\in \calp_A, t(q)=v \text{ or } h(q)= v} c^v_q q.$
Since $\calp_A=\emptyset$ in this case,  $D(v)=0$ for all $v\in V$.
Since $V$ is a basis of $\call$, $D$ is the zero map.

Conversely, Suppose $\Gamma$ contains an arrow $p_0$. Then we have the derivation $D_{p_0,p_0}$ by Theorem~\mref{thm:rsdiff}. Since $D_{p_0,p_0}(p_0)=p_0$ which is nonzero, we have obtained a non-zero derivation on $\bfk\Gamma$.
\end{proof}

\subsection{A canonical basis of $\diff(\bfk\Gamma)$}

We now display a canonical basis of the Lie algebra $\diff(\bfk\Gamma)$ for a quiver $\Gamma$. For a given $s\in \calp$, we have the inner derivation
\begin{equation}
D_s:\bfk\Gamma\to \bfk\Gamma, \quad D_s(q)=sq-qs, \quad \forall q\in \calp.
\notag %\mlabel{eq:vbase}
\end{equation}

\begin{theorem}
Let $\Gamma$ be a quiver.
A basis of the Lie algebra $\diff(\bfk\Gamma)$ is given by the set
\begin{equation}
\frak B:=\frakBa \cup \frakBal
\notag %\mlabel{eq:dbase}
\end{equation}
where
\begin{equation}
\frakBa:=\{D_s\ |\ s\in \calpa\} \quad \text{and} \quad
\frakBal:=\{D_{r,s}\ |\ r\in E ,s\parallel r\}. \mlabel{eq:baal}
\end{equation}
Thus
$\diff(\bfk\Gamma)=\frakDa\oplus\frakDal$ where $\frakD_i$ are the
$\bf k$-linear space with bases $\frakB_i$ for $i=1,2$.
\mlabel{thm:dbase}
\end{theorem}

We will call $\frakB$ the {\bf canonical basis} of $\diff(\bfk\Gamma)$.
\begin{proof}
Since the operators in $\frakB$ are derivations by Theorem~\mref{thm:rsdiff}, we only need to verify that the operators in $\frakB$ are linearly independent and that $\frakB$ spans the whole space of derivations.

\smallskip

\noindent
{\bf Step 1. $\frakB$ is linearly independent. }
Suppose there are $c_s, c_{r,s}\in \bfk$ such that
$$D:= \sum_{s\in \calpa} c_s D_s
+ \sum_{r\in E,s\parallel r}
    c_{r,s} D_{r,s}=0.    $$
Then for any given $s_0\in \calpa$, by the definitions of $D_s$ and $D_{r,s}$ we have
\begin{eqnarray*}
0&=& D(h(s_0))\\
&=& \sum_{s\in \calpa} c_s D_s(h(s_0)) \\
&=& \sum_{s\in \calpa} c_s \big(s\,h(s_0) - h(s_0)s\big) \\
&=& \sum_{s\in \calpa,\;h(s)=h(s_0)} c_s s -\sum_{s\in
\calpa,\; t(s)=h(s_0)} c_s s
\end{eqnarray*}
Since $h(s)\neq t(s)$ in the sums, the index sets of the two sums
are disjoint. Thus both the sums equal to zero and hence $c_s=0$ for
all $s\in \calpa$ with $h(s)=h(s_0)$. In particular, $c_{s_0}=0$. Thus,
$$D= \sum_{r\in E, s\parallel r}
    c_{r,s} D_{r,s}.  $$
Further, for any given $r_0\in E$ and $s_0\parallel r_0$, by the definitions of $D_s$ and
$D_{r,s}$ we have $ 0= D(r_0) =\sum\limits_{s_0\neq s\in \calp,
\; s\parallel r_0}
    c_{r_0,s} s +c_{r_0,s_0} s_0,    $
and hence $c_{r_0,s_0}=0$.

Thus we have proved that $\frakB$ is linearly independent.

\smallskip

\noindent
{\bf Step 2. $\frakB$ is a spanning set of derivations on $\bfk\Gamma$. }
Let $D: \bfk\Gamma \to \bfk\Gamma$ be a given derivation. Then $D$ is defined by Eq.~(\mref{eq:vdiff2}) and Eq.~(\mref{eq:pdiff2}) in Corollary~\mref{co:dcond2}.
In particular, for $p\in E$, by Eq.~(\mref{eq:pdiff2}) we have
\begin{equation}
D(p)=  \sum_{q\in \calpa, t(q)=h(p)}
c^{t(q)}_q pq+  \sum_{q\parallel p}  c^{p}_{q} q
-  \sum_{q\in \calpa, h(q)=t(p)}
 c^{t(q)}_q qp \mlabel{eq:pdiff4b}
\end{equation}
for certain coefficients $c^{t(q)}_q\in \bfk$ where $q\in \calp_A$ with $t(q)=h(p)$ and $c^p_q\in \bfk$ where $q\parallel p$.

We claim that $D$ agrees with the operator $\bar{D}$ defined by the linear combination
\begin{equation}
    \bar{D}= -\sum_{s\in \calpa} c^{t(s)}_s D_s + \sum_{r\in
E, s\in \calp,\; h(s)=h(r), t(s)=t(r)} c^r_s D_{r,s},
\notag %\mlabel{eq:bard}
\end{equation}
obtained by the same coefficients appeared in Eq.~(\mref{eq:pdiff4b}).
As a linear combination of derivations, $\bar{D}$ is also
a derivation. Any path in $\calp$ is either a vertex or a
product of arrows. Thus by the product rule of derivations, to show the equality of $D$ and $\bar{D}$, we only need
to verify that $D(q)=\bar{D}(q)$ for each $q=v\in V$ and $q=p\in E$.

First let $q=v\in V$. Since $D_{r,s}(v)=0$,  we have
$$\bar{D}(v)= -\sum_{s\in \calpa} c^{t(s)}_s D_s(v)
= -\sum_{s\in \calpa} c^{t(s)}_s (sv-vs)
= -\sum_{s\in \calpa,
h(s)=v} c^{t(s)}_s sv + \sum_{s\in \calpa, t(s)=v} c^{t(s)}_s vs. $$
Here the last equality follows since, in the first sum, $sv$ is $s$ if $h(s)=v$ and is
zero otherwise, and in the second sum, $vs$ is $s$ if $t(s)=v$ and is zero otherwise. Thus $\bar{D}(v)$ agrees with
$D(v)$ as defined in Eq.~(\mref{eq:vdiff2}).

Next let $q=p\in E$.
Then $D_{r,s}(p)$ is $s$ if $r=p$ and is 0 otherwise. Thus we have

\begin{eqnarray*}
\bar{D}(p)&=& -\sum_{s\in \calpa} c^{t(s)}_s D_s(p)
+ \sum_{r\in E, s\parallel r} c^{r}_{s} D_{r,s}(p)\\
&=& -\sum_{s\in \calpa} c^{t(s)}_s (-ps+sp)
+ \sum_{s\parallel p} c^{p}_{s} s\\
&=& \sum_{s\in \calpa, t(s)=h(p)} c^{t(s)}_s ps- \sum_{s\in \calpa,
h(s)=t(p)} c^{t(s)}_s sp + \sum_{s\parallel p}
c^{p}_{s} s.
\end{eqnarray*}
This agrees with $D(p)$ in
Eq.~(\mref{eq:pdiff4b}). It means $\bar{D}=D$, showing that $\frakB$
spans $\diff(\bfk\Gamma)$.

The proof of Theorem~\mref{thm:dbase} is completed. \end{proof}

\subsection{Structure theorem of $\diff(\bfk\Gamma)$}

\delete{
Let $C(\bfk\Gamma)$ be the center of  $\bfk\Gamma$. A quiver $\Gamma$ is called an {\bf oriented loop} if it has a unique vertex $v$ and a unique arrow from $v$ to $v$. By~\mcite{CB},
we have
\begin{equation}
C(\bfk\Gamma)=\left \{\begin{array}{ll} \bfk[T], & \Gamma \text{ is\ an\ oriented\ loop}\ T, \\
\bfk, & \text{ otherwise.}
\end{array} \right .
\notag %\mlabel{eq:center}
\end{equation}
Thus we obtain the following consequence of Proposition~\mref{pp:diff}.
\begin{coro}
If a quiver $\Gamma$ is not an oriented loop, then we have a Lie algebra embedding
\begin{equation}
\bfk\Gamma/\bfk \cong \indiff(\bfk\Gamma)\hookrightarrow \diff(\bfk\Gamma).
\mlabel{eq:embed}
\end{equation}
\mlabel{co:qcent}
\end{coro}
When $\Gamma$ is an oriented loop, $C(\bfk\Gamma)=\bfk\Gamma$. So in this case $\im \frakD= \indiff(A)$ is equal to zero.
}

Let $p\in E$ and let $\sum_{i=1}^k c_i q_i\in \bfk\Gamma$ with $c_i\in\bfk$ and $q_i\in \calp$.
We denote
\begin{equation}
 D_{p,\sum_{i=1}^k c_i q_i}= \sum_{i=1, q_i\parallel p}^k c_i D_{p,q_i}.
 \mlabel{eq:lie40}
 \end{equation}

\begin{theorem}
{\bf (Basis Theorem of $\diff(\bfk\Gamma)$)} 
For derivations in $\diff(\bfk\Gamma)$, the following relations hold.
\begin{eqnarray}
[D_p,D_r] &=&D_{[p,r]},\ \ \ \ \text{for} \quad p,r\in \calp, \mlabel{eq:lie1} \\
{}  [D_p,D_{r,s}] &=& D_{D_{r,s}(p)}, \ \ \ \ \text{for} \quad r\in E, p,s\in \calp, s\parallel r,  \mlabel{eq:lie2} \\
{}  [D_{r,s},D_{p,q}] &=& D_{p,D_{r,s}(q)}-D_{r,D_{p, q}(s)}, \ \ \ \ \text{for}
\quad r,p\in E, s\parallel r, q\parallel p. \mlabel{eq:lie3}
\end{eqnarray}
\mlabel{thm:lie}
\end{theorem}
\begin{proof}
Eq.~(\mref{eq:lie1}) follows the fact that the map in Eq.~(\mref{eq:indh}) is a Lie algebra homomorphism.

Next let $r\in E, p,s\in \calp$ with $h(s)=h(r)$ and $t(s)=t(r)$.
Then for $t\in \calp$, we have
\begin{eqnarray*}
[D_p,D_{r,s}](t)&=& D_pD_{r,s}(t) - D_{r,s}D_p(t)\; =\;
D_p(D_{r,s}(t)) - D_{r,s}(pt-tp) \\
&=& pD_{r,s}(t) - D_{r,s}(t)p - \big( D_{r,s}(p)t +pD_{r,s}(t) - (D_{r,s}(t)p+tD_{r,s}(p))\big) \\
&=& -D_{r,s}(p)t + tD_{r,s}(p) \\
&=&-D_{D_{r,s}(p)}(t).
\end{eqnarray*}
This proves Eq.~(\mref{eq:lie2}).

Finally let $r,p\in E$ and $s\parallel r$ and $q\parallel p$. Since both sides of Eq.~(\mref{eq:lie3}) are derivations, by the product rule of derivations, we only need to prove that, for $t\in E$, the following holds
\begin{equation}
[D_{r,s},D_{p,q}](t) = D_{p,D_{r,s}(q)}(t)-D_{r,D_{p,q}(s)}(t).
\mlabel{eq:lie31}
\end{equation}
Let such a $t$ be given. If $t\neq r,p$, then both sides of Eq.~(\mref{eq:lie31}) are zero. If $t=p$, then both sides of Eq.~(\mref{eq:lie31}) equal to $D_{r,s}(q)-D_{p,q}(s)$ if $t=r$ and equal to $D_{r,s}(q)$ if $t\neq r$. If $t=r$, then both sides of Eq.~(\mref{eq:lie31}) equal to $D_{r,s}(q)-D_{p,q}(s)$ if $t=p$ and equal to $-D_{p,q}(s)$ if $t\neq p$. This proves Eq.~(\mref{eq:lie3}).
\end{proof}

 Note that we usually
do not require $D_p$ or $D_r$ to be in $\frakBa$. So $h(p)=t(p)$ or
$h(r)=t(r)$ are allowed. In fact, even when $D_p$ and $D_r$ are in
$\frakBa$, $D_{[r,p]}=D_{rp}-D_{pr}$ might not be in the linear
space $\frakDa$ spanned by $\frakBa$. For example, if $r$ is a path
from a vertex $v_1$ to another vertex $v_2\neq v_1$ and $p$ is a
path from $v_2$ to $v_1$, then $D_r$ and $D_p$ are in $\frakBa$. But
$pr$ and $rp$ are both oriented cycles, so $D_{pr}$ and $D_{rp}$ are
not in $\frakBa$.

\delete{
We now apply Theorem~\mref{thm:lie} to study the Lie algebra of derivations on a path algebra as the quiver varies.

\begin{prop}
Let $\Gamma_1$ be a quiver and let $\Gamma_2$ be a subquiver of $\Gamma_1$. Then there is a canonical embedding of $\diff(\bfk\Gamma_2)$ as a Lie subalgebra of $\diff(\bfk\Gamma_1)$.
\mlabel{pp:sublie}
\end{prop}

\begin{proof}
We let $\frakB^{(1)}$ and $\frakB^{(2)}$ denote the canonical basis of $\diff(\bfk\Gamma_1)$ and $\diff(\bfk\Gamma_2)$ respectively. Then we have a natural embedding of $\frakB^{(2)}$ into $\frakB^{(1)}$ sending an operator $D_s$ for $s\in \calp_{\Gamma_2}$ with $h(s)\neq t(s)$ or a $D_{r,s}$ for $r,s\in \calp_{\Gamma_2}$ with $\ell(r)=1$, $h(r)=h(s)$ and $t(r)=t(s)$ to the operator defined in the same way, but defined on $\bfk\Gamma_1$. Expanding by $\bfk$-linearity, this natural embedding gives a natural embedding of $\diff(\bfk\Gamma_2)$ into $\diff(\bfk\Gamma_1)$. By Theorem~\mref{thm:lie}, the Lie brackets on the two Lie algebras are the same. Thus the latter embedding is a Lie algebra homomorphism, proving the proposition.
\end{proof}
}

\subsection{Ideals and nilpotency}
We next apply Theorem~\mref{thm:lie} to study Lie algebra properties of $\diff(\bfk\Gamma)$.

\begin{prop}
For any non-trivial quiver $\Gamma$, the Lie algebra $\diff(\bfk\Gamma)$ is not nilpotent.
\mlabel{pp:nonnil}
\end{prop}
\begin{proof}
By the well-known Engel theorem~\mcite{Hu}, a Lie algebra $\frakg$ is nilpotent if and only if for all its elements $g$, the adjoint derivation
$${\bf ad}\,g:  \frakg\to \frakg, \quad h \mapsto [g,h], \quad h \in \frakg, $$
is nilpotent.
Let $p$ be an arrow in $\Gamma$.
Then we have ${\bf ad}D_{p,p}(D_{p})=D_{p}$ and thus for any natural number $n$, $({\bf ad}D_{p,p})^n(D_{p})=D_{p}$. So ${\bf ad}D_{p,p}$ is not nilpotent.
\end{proof}

%It is known that nilpotent algebras are always solvable. So, it still needs to discuss the solvability of the Lie algebra $\diff\bfk\Gamma$ in the further work.

\begin{theorem}
Let $\Gamma$ be an acyclic quiver. Let $\frakBa, \frakBal\subseteq
\frakB$ and $\frakDa=\bfk \frakBa$,  $\frakDal=\bfk \frakBal$ be
defined in Theorem~\mref{thm:dbase}. Then,
\begin{enumerate}
\item $\frakDa$ (resp. $\frakDal$) is an ideal (resp. subalgebra) of the Lie algebra
$\diff(\bfk\Gamma)$;
\mlabel{it:ideal}
\item
{\bf (Structure of $\diff(\bfk\Gamma)$)}
$\diff(\bfk\Gamma)$ is a semi-direct sum  of the Lie ideal $\frakDa$
and the Lie subalgebra $\frakDal$, that is,
$$\diff(\bfk\Gamma)=\frakD_1\ltimes\frakD_2;$$ \mlabel{it:sum}
\item
If $\Gamma$ is also a finite quiver, then $\frakDa$ is a nilpotent
Lie algebra. \mlabel{it:nil}
\end{enumerate}
\mlabel{thm:difb}
\end{theorem}

\begin{proof}
(\mref{it:ideal}).
Since $\Gamma$ does not contain any oriented cycles, for any $p,r\in \calpa$, we have $rp\in \calpa$ unless
$rp=0$ and $pr\in\calpa$ unless $pr=0$. Thus by Eq.~(\mref{eq:lie1}), $\frakDa$ is closed under the Lie bracket. Further, for $p\in \calpa$ and $r\in E, s\parallel r$, by the acyclicity of $r$ and the definition of $D_{r,s}$ in Eq.~(\mref{eq:rsrec}), $D_{r,s}(p)=\sum_i c_i q_i$ for $q_i\in \calp\backslash V$. Since $\calp\backslash V=\calpa$ by assumption, we see that
$D_{D_{r,s}(p)}=\sum_i c_i D_{q_i}$ is in $\frakDa$. Thus
$[D_{r,s},D_p]$ is in $\frakDa$ by Eq.~(\mref{eq:lie2}). Since
$\diff(\bfk\Gamma)=\frakDa\oplus\frakDal$ by
Theorem~\mref{thm:dbase}, this proves that $\frakDa$ is an ideal of
$\diff(\bfk\Gamma)$.

By Eq.~(\mref{eq:lie3}), $\frakDal$ is a Lie subalgebra of
$\diff(\bfk\Gamma)$.
\smallskip

\noindent
(\mref{it:sum}). This follows from Item~(\mref{it:ideal}) and Theorem~\mref{thm:dbase}.
\smallskip

\noindent
(\mref{it:nil}).
We first note that the minimal length of $\calp':=\calp\backslash V$ is one. We then note that, for $$\calp^{(2)}:=[\calp',\calp']:=\{[p,q]=pq-qp \,|\, p,q\in \calp'\},$$
the minimal length is two unless $[\calp',\calp']=0$. Let $\frakg:=\bfk (\calp\backslash V)$. Since $\calp^{(2)}\subseteq \frakg$, $\frakg$ is a Lie subalgebra of $\bfk \Gamma$. By an inductive argument, we see that, for the recursively defined
$\calp^{(n+1)}:=[\calp',\calp^{(n)}],$
its minimal length is $n+1$ unless $\calp^{(n+1)}=0$.

On the other hand, by our acyclicity and finiteness assumptions on
$\Gamma$, the lengths of paths in $\Gamma$ is bounded by $|V|$:
suppose there is a path $p$ of length $|V|+1$ with its standard
decomposition $p=v_0p_1v_1\cdots v_{|V|} p_{|V|+1} v_{|V|+1}$. Then
there are $0\leq i<j\leq |V|+1$ such that $v_i=v_j$. This shows that
$p$ contains an oriented cycle, contradicting the acyclicity
assumption.

Combining the above two points, we see that $\calp^{(n)}=0$ for large enough $n$.
Thus $\frakg$ is nilpotent.

Under the acyclic assumption, we have $\calp\backslash V=\calpa$. So
the Lie algebra homomorphism $\cald: \bfk\Gamma \to
\diff(\bfk\Gamma)$ from Eq.~(\mref{eq:indh}) sends the above Lie
algebra $\frakg$ to $\frakD_A=\bfk\frakBa$ surjectively. Thus
$\frakD_A$ is nilpotent.
\end{proof}

We note that when the restriction that $\Gamma$ is acyclic  is removed, the first statement of Theorem~\mref{thm:difb} is no longer true. This is because in Eq.~(\mref{eq:lie1}): $[D_p,D_r] =D_{[p,r]}$, the right hand side might not be in $\frakBa$ even if $D_p$ and $D_r$ are. See the remark and example after Theorem~\mref{thm:lie}.

\subsection{Inner derivations and the canonical basis}

We now express inner derivations in $\indiff\bfk\Gamma$ in terms of the canonical basis.

\begin{prop}
Let $q\in \calp$ be such that $h(q)=t(q)$. Let $v_0=h(q)$. We have
\begin{equation}
D_q= \sum_{p\in E, t(p)=v_0} D_{p,qp} - \sum_{r\in E,h(r)=v_0} D_{r,rq}.
\mlabel{eq:innerq}
\end{equation}
\mlabel{pp:basis2}
\end{prop}

%The proposition implies that the Lie ideal $\indiff(\bfk\gamma)$ is in the Lie subalgebra $\frakDal$.

\begin{proof}
Note that both sides of the equation are derivations and $\bfk\Gamma$ is generated by $V\cup E$ as a $\bfk$-algebra. So by the product formula of derivations, we only need to verify that the two sides agree when acting on $V$ and $E$.

For $v\in V$, we have
$$D_q(v)=qv-vq = \left\{\begin{array}{ll} q-q=0, & v= v_0, \\ 0-0=0, & v\neq v_0. \end{array} \right .$$
Also
$$ \sum_{p\in E, t(p)=v_0} D_{p,qp}(v) - \sum_{r\in E,h(r)=v_0} D_{r,rq}(v)=0$$
by the definition of $D_{r,s}$. So we are done in this case.

For $s\in E$, we have
$$ D_q(s)=\left\{\begin{array}{ll} qs-sq, & t(s)=v_0=h(s), \\
qs, & t(s)=v_0\neq h(s), \\
-sq, & t(s)\neq v_0=h(s), \\
0, & t(s)\neq v_0\neq h(s). \end{array}\right .
$$
On the other hand, we have
\begin{eqnarray*}
\sum_{p\in E, t(p)=v_0}D_{p,qp}(s)&=&\left\{\begin{array}{ll}
qs, & s=p \text{ for some } p\in E \text{ with }t(p)=v_0, \\ 0, & \text{otherwise}. \end{array} \right.\\
&=&\left\{\begin{array}{ll}
qs, & t(s)=v_0, \\ 0, & t(s)\neq v_0. \end{array} \right . \\
-\sum_{r\in E, h(r)=v_0}D_{r,sq}(s)&=&\left\{\begin{array}{ll}
-sq, & s=r \text{ for some } r\in E \text{ with }h(r)=v_0, \\ 0, & \text{otherwise}. \end{array} \right.\\
&=& \left \{ \begin{array}{ll}
-sq, & h(s)=v_0, \\ 0, & h(s)\neq v_0. \end{array} \right. 
\end{eqnarray*}
Thus the actions of the two sides of Eq.~(\mref{eq:innerq}) on $E$ agree.
\end{proof}

\section{Combinatorial derivations and their relations}
\mlabel{sec:comb}
In this section, we study combinatorial derivations on a path algebra $\bfk\Gamma$, namely derivations from the combinatorial objects of vertices, edges and faces of $\Gamma$. We define various relation matrices, study their ranks and obtain dimensional formulas of these derivations. These dimensional formula will be applied in the next section to give a strengthened form of Euler's polyhedron formula and to determine the structure of the Lie algebra $HH^1(\bfk\Gamma)$.

\subsection{Combinatorial derivations and their relation matrices}
We will consider a {\bf quiver  $\Gamma$ of genus $g$} which is
defined to be a quiver together with a fixed embedding of $\Gamma$
into a surface  $S$ of genus $g$ such that $g$ is smallest. Such a quiver is called a {\bf topological quiver}.
A quiver $\Gamma$ is
called {\bf connected} if the underlying set of $\Gamma$ is
connected.
The set $F$ of {\bf faces} of $\Gamma$ is the set of connected components of $S\backslash \Gamma$, or more precisely the complement of the underlying set of $\Gamma$ in $S$.

When the genus $g$ is zero, we can take the surface $S$ to be the
Riemann sphere, through the stereographic projection from
the Riemann sphere to $\RR^2$.

\subsubsection{Combinatorial derivations from a quiver}
We put together various combinatorially defined derivations on the path algebra of a quiver $\Gamma=(V,E)$. We recall the following notations.
\begin{enumerate}
\item
For $v\in V$, we call $D_v$ a {\bf vertex derivation} and let $\frakD_V$ denote the linear space spanned by $\{D_v\,|\,v\in V\}$, called the {\bf space of vertex derivations}.
\item
For $p\in E$, we call $D_{p,p}$ an {\bf edge derivation} and let $\frakD_E$ denote the linear space spanned by $\{D_{p,p}\,|\,p\in E\}$, called the {\bf space of edge derivations}.
\end{enumerate}

For a face $f\in F$, let $\pcyc_f$ be the boundary of $f$. It is an unoriented cycle of $\Gamma$ consisting of arrows that are not necessarily in one direction, called a {\bf primitive cycle} of $\Gamma$. Thus the set $C:=C_\Gamma$  of primitive cycles of $\Gamma$ is in bijection with the set $F$ of faces of $\Gamma$. For a primitive cycle $\pcyc=\pcyc_f$ of a face $f$, we will define a {\bf face derivation} $D_\pcyc:=D_{\pcyc_f}$ on $\bfk\Gamma$. First define
\begin{equation}
D_{\pcyc}(v)=0, v\in V;\quad
D_{\pcyc}(p)=\left\{\begin{array}{ll} p, & p \text{ is clockwise on } \pcyc,\\
-p, & p \text{ is counterclockwise on }\pcyc,\\
0, & p \text{ is not on } \pcyc.
\end{array} \right . \ p\in E.
\notag %\mlabel{eq:fdiff}
\end{equation}
Here being clockwise or counterclockwise is viewed from inside the face for the primitive cycle.
We then expand $D_{\pcyc}$ to $\bfk\Gamma$
by the product rule, noting that $\bfk\Gamma$ is the algebra generated by $V\cup E$.

If a cycle is shared by two faces, such as the quiver of one oriented loop, there will be two face derivations from the two faces.
Also, if $p$ is an edge in the interior of $\pcyc$, then
$D_{\pcyc}(p)=0$.
The name face derivation is justified by the following alternative description of $D_\pcyc$.
\begin{lemma}
Let a primitive cycle $\pcyc\in C_\Gamma$ be comprised of an
ordered list of arrows $p_{1},\cdots, p_{s}\in E$. Then
\begin{equation}
D_{\pcyc}= \pm D_{p_1,p_1}\pm \cdots \pm D_{p_s,p_s},
\mlabel{eq:primd}
\end{equation}
where a $\pm D_{p_i,p_i}$ is $D_{p_i,p_i}$ if $p_i$ is in clockwise direction when viewed from
 the interior of the face of $\pcyc$ and is $-D_{p_i,p_i}$
 otherwise.
In particular, $D_{\pcyc}$ is a derivation.
\mlabel{lem:fdiff}
\end{lemma}
\begin{proof}
We only need to check that the two operators agrees on $V\cup E$.
But this is clear from the definition of $D_{\pcyc}$.
\end{proof}

Let $\frakD_F$ denote the linear span of
$\{D_{\pcyc_f}\,|\,f\in F\}$, called the {\bf space of face derivations}.

Since the concept of a face derivation will be important
for the rest of the paper, we make the following remarks and
illustrate their contents by the following quiver $\Gamma$.
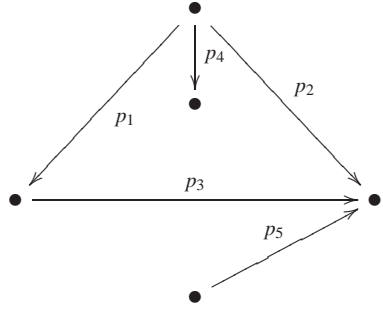
\begin{equation}
\xymatrix{ && \bullet \ar^{p_1}[ddll] \ar^{p_2}[ddrr] \ar^{p_4}[d] && \\
&& \bullet && \\
\bullet \ar^{p_3}[rrrr] &&&& \bullet \\
&& \bullet \ar^{p_5}[rru]
}
%\xymatrix{ B \ar[rr]^{j_B}\ar[drr]^{f} && \ncshao(B) \ar[d]_{\free{f}} \\ && A}
\mlabel{eq:eg}
\end{equation}
The quiver $\Gamma$ has two primitive cycles: the cycle $\pcyc_1$ of the finite face of $\Gamma$ and the cycle $\pcyc_0$ of the infinite face of $\Gamma$.
\begin{remark}
{\rm
\begin{enumerate}
\item In the case a quiver $\Gamma$ is planar, for the boundary $\pcyc_0$ of the infinite face, an arrow on $\pcyc_0$ is in clockwise direction when viewed from the interior of the infinite face means that the arrow is in counterclockwise direction when viewed from the interior of the quiver.
For the quiver in the diagram~(\mref{eq:eg}), the arrow $p_2$ is
clockwise for the primitive cycle from the finite face, but is
counterclockwise for the primitive cycle from the infinite face.
\mlabel{it:inf}
\item If an arrow $p$ is on the boundary of two faces of $\Gamma$, then $p$ will be in the clockwise direction on one boundary and in the counterclockwise direction on the other. Thus $D_{p,p}$ will have a plus sign in $D_\pcyc$ for the primitive cycle $\pcyc$ of one boundary and will have a minus sign in $D_\pcyc$ for the other.
For our example of $\Gamma$, $D_{p_1,p_1}$ has a minus sign in $\pcyc_1$ and a plus sign in $\pcyc_0$.
\mlabel{it:p2}
\item If an arrow $p$ is not on the boundary of two faces of $\Gamma$, then both sides of the arrow are in the same face of $\Gamma$. In other words, $p$ will appear twice in $D_\pcyc$ for the primitive cycle $\pcyc$ of this face, once in the clockwise direction, once in the counterclockwise direct. As a results, $D_{p,p}$ will appear exactly twice in $D_\pcyc$, once with a positive sign and once with a negative sign. Consequently, there will be no contribution of $D_{p,p}$ in $D_\pcyc$.
For our example of $\Gamma$, both sides of $p_4$ are in the finite face of $\Gamma$. The primitive cycle $\pcyc_1$ of this face gives
$$D_{\pcyc_1}=-D_{p_1,p_1}+D_{p_4,p_4}-D_{p_4,p_4} +D_{p_2,p_2}-D_{p_3,p_3}=-D_{p_1,p_1} +D_{p_2,p_2}-D_{p_3,p_3}.$$
So $D_{p_4,p_4}$ does not contribute to $D_{\pcyc_1}$. Likewise, $D_{p_5,p_5}$ does not contribute to $D_{\pcyc_0}$ from the infinite face.
\mlabel{it:p1}
\item The previous remark applies in particular when a quiver contains only a unique primitive
cycle $\pcyc_0$ (in the case  the quiver is planar, the unique
primitive cycle is just the boundary of  the infinite face, which is
not proper). This is because for such a quiver, no arrow can appear
on the boundary of two faces. Thus for such a quiver, we have
$D_{\pcyc_0}=0$. \mlabel{it:ponly}
\item By Remark~\mref{rk:cycle}.(\mref{it:p2}) and (\mref{it:p1}), each arrow $p\in E$ will appear in exactly one $D_\pcyc$ with a plus sign and in exactly another (or the same) $D_\pcyc$ with a minus sign. Consequently, we have
\begin{equation}\sum_{i=1}^{\gamma_2}D_{\pcyc_i}=0.
\mlabel{eq:zero}
\end{equation}
\mlabel{it:ppm}
\item For an oriented cycle $q$ of $\Gamma$, let $\pcyc$ be the corresponding primitive cycle. Then the $D_\pcyc$ defined here is different with $D_q$ given in Eq.~(\mref{eq:innerq}).
\end{enumerate}
}
\mlabel{rk:cycle}
\end{remark}

\subsubsection{Relation matrices}

Denote $\gamma_0=|V|$, $\gamma_1=|E|$, $\gamma_2=|F|=|C|$. They are all finite since $\Gamma$
is finite. With these notations, we will use the following enumerations of sets.
\begin{equation}
V=\{v_i\ |\ 1\leq i\leq \gamma_0\},\quad E=\{p_k\ |\ 1\leq k\leq
\gamma_1\},\quad C=\{\pcyc_j\ |\ 1\leq j\leq
\gamma_2\}.
\mlabel{eq:list}
\end{equation}
By Eq.~(\mref{eq:innerq}) and Eq.~(\mref{eq:primd}), we have the
following system of linear relations.

\begin{eqnarray}
D_{v_i}&=& \sum_{k=1}^{\gamma_1} c_{i,k} D_{p_k,p_k}=\sum_{p\in E,t(p)=v_i} D_{p,p} - \sum_{r\in E, h(r)=v_i} D_{r,r}, \quad 1\leq i\leq
\gamma_0, \mlabel{eq:vp}
\\
D_{\pcyc_j}&=& \sum_{k=1}^{\gamma_1} c_{\gamma_0+1+j, k}
D_{p_k,p_k}= \sum_{k=1}^{\gamma_1} d_{j,k} D_{p_k,p_k}, \quad 1\leq
j\leq \gamma_2.
\mlabel{eq:pp}
\end{eqnarray}
where all $d_{j,k}=\pm 1$, or\; $=0$. Thus by
Lemma~\mref{lem:fdiff}, we find that $\frakD_F$ is a subspace of
$\frakD_E$.

\begin{defn}
{\rm
\begin{enumerate}
\item
Define the {\bf (differential) vertex-arrow matrix} $C_{va}$ of
$\Gamma$ to be the coefficient matrix of the linear system in Eq.~(\mref{eq:vp}).
\item
Define the {\bf (differential) cycle-arrow matrix} $C_{ca}$ of
$\Gamma$ to be the coefficient matrix of the linear system in Eq.~(\mref{eq:pp}).
\item
Define the {\bf \conmat} of $\Gamma$ to be the
coefficient matrix $C_\Gamma$ of the combined linear system in
Eq.~(\mref{eq:vp}) and Eq.~(\mref{eq:pp}), that is, the $(\gamma_0+\gamma_2)\times \gamma_1$-matrix
$\left[\begin{array}{c}C_{VP}\\C_{CP} \end{array}\right ]$.
\item
Define the {\bf \bounmat} of a quiver $\Gamma$ is the $\gamma_2\times \gamma_2$ matrix $B_\Gamma=[e_{j,r}]_{0\leq j,r\leq \gamma_2-1}$ in which
$e_{j,j}$ is the number of arrows on $\pcyc_j$ that are also on $\pcyc_r$ for some $r\neq j$, and $-e_{j,r}$ for $r\neq j$ is the number of arrows on $\pcyc_j$ that are also on $\pcyc_r$.
\end{enumerate}
}
\mlabel{de:boundary}
\end{defn}

\begin{remark}
{\rm
\begin{enumerate}
\item
The matrix $C_{va}$ encodes the relationship
between the vertex derivations $D_{v_i},$ $1\leq i\leq \gamma_0$ and
the edge derivations $D_{p_j,p_j}, 1\leq j\leq \gamma_1$. Since the edge derivations $D_{p_k,p_k}, 1\leq k\leq \gamma_1,$ are linearly independent the rank of the row space of $C_{va}$ is $\dim\frakDv$.
\item
The matrix $C_{ca}$
encodes the relationship between the primitive cycle derivations
$D_{\pcyc_k}, 0\leq k\leq \gamma_2-1$ and the arrow derivations
$D_{p_j,p_j}, 1\leq j\leq \gamma_1$. The rank of the row space of $C_{ca}$ is just $\dim\frakD_F$.
\item
By the definition of $e_{j,j}$ and $e_{j,r}$ for $r\not=j$, the
matrix $B_\Gamma$ is independent of the direction of the quiver $\Gamma$ and
only depend on the underlying graph.
\end{enumerate}
}
\end{remark}

In order to study the rank of
the \bounmat, we introduce preparatory concepts and results of
matrices.

\begin{defn}
{\rm
\begin{enumerate}
\item
A square matrix $M$ over a number field is called
{\bf irreducible} if $M$ cannot be
written as a block matrix $M=\left[ \begin{array}{cc} M_1 & O\\
M_3 & M_2\end{array} \right]$ where $O$ is a zero matrix and $M_1$
and $M_2$ are both square matrices.
\item
An $n\times n$-matrix $M=[m_{i,j}]$ is called {\bf
weakly diagonally dominant} if
$|m_{i,i}|\geq\sum^n_{j=1,j\not=i}|m_{i,j}|$
 for
$i=1,\cdots,n,$ with strict inequality for at least one $i$.
\end{enumerate}
}
\end{defn}

\begin{theorem}$($\cite[Section 10.7]{LT}$)$ Let $M=[m_{ij}]$ be an $n\times n$ irreducible matrix over a number field. If
$M$ is weakly diagonally dominant,
 then $M$ is invertible.
\mlabel{thm:lt}
\end{theorem}

Using of this fact, we prove the following
\begin{lemma}
Let $M=[m_{ij}]$ be an $n\times n$ irreducible matrix with entries
in a number field. If, for each $i=1,\cdots,n$,
$\sum_{j=1}^{n}m_{i,j}=0$, $m_{i,i}>0$ and $m_{i,j}\leq 0$ for
$j\not=i$, then the rank of $M$ is $n-1$. \mlabel{lem:rk}
\end{lemma}
\begin{proof}
Since $\sum_{j=1}^{n}m_{i,j}=0$ for each $i=1,\cdots,n$, the sum of the $n$ column vectors of $M$ is zero. Thus the rank of $M$ is less than or equal to $n-1$.

On the other hand,  consider the $(n-1)\times (n-1)$-submatrix $M_1$
consisting of the first $n-1$ rows and columns of $M$, that is $M_1=[m_{i,j}]_{1\leq i,j\leq n-1}$.
Since $\sum_{j=1}^{n}m_{i,j}=0$ for each $i=1,\cdots,n-1$, we have
$$m_{i,i}=-\sum^n_{j=1,j\not=i}m_{i,j} =-\sum^{n-1}_{j=1,j\not=i}m_{i,j}-m_{i,n} \geq -\sum^{n-1}_{j=1,j\not=i}m_{i,j} =\sum^{n-1}_{j=1,j\not=i}|m_{i,j}|.$$

Since $M$ is irreducible and symmetric, there is at least one
$m_{i_0,n}\neq 0$ in the last column of $M$ other than $m_{n,n}$.
Otherwise we would have $M=\left[\begin{array}{cc} M_1 & \vec{0} \\
M_0 & m_{n,n} \end{array} \right ] $ for the zero vector $\vec{0}$,
contradicting the irreducibility condition on $M$. Then, since
$m_{i,j}\leq 0$ for $i\not= j$ by assumption, we have $m_{i_0,n}<
0$. Therefore,
$$m_{i_0,i_0} =-\sum^{n-1}_{j=1,j\not=i_0}m_{i_0,j} -m_{i_0,n}>
\sum^{n-1}_{j=1,j\not=i_0}|m_{i_0,j}|.$$ Thus,  $M_1$ is a weakly
diagonally dominant matrix. Then by Theorem~\mref{thm:lt}, $M_1$ is
invertible. Together with the observation made at the beginning of
the proof, we conclude that $rk(M)=n-1$.
\end{proof}

\subsubsection{Ranks of relation matrices}

\begin{theorem}
\begin{enumerate}
\item
The vertex-arrow matrix $C_{va}$ has rank $\gamma_0-1$.
\mlabel{it:a}
\item
The cycle-arrow matrix $C_{ca}$ has rank $\gamma_2-1$.
\mlabel{it:b}
\item
Let $\Gamma$ be a connected finite quiver and suppose that the
ground field $\bfk$ has characteristic $0$. Then the rank of the
\bounmat $B_\Gamma$ is $\gamma_2-1$.
\mlabel{it:c}
\end{enumerate}
\mlabel{thm:rank}
\end{theorem}

\begin{proof}
(\mref{it:a})
Since $e=\sum_{i=1}^{\gamma_0} v_i$ is the identity of $\bfk\Gamma$,
we have
\begin{equation}
D_e=\sum_{i=1}^{\gamma_0} D_{v_i}=0. \mlabel{eq:v1}
\end{equation}

If $\gamma_0=1$, then $V=\{v_1\}$ and $D_{v_1}=0$ by Eq.~(\mref{eq:v1}). Thus
$C_{va}=0$ and its rank is $0=\gamma_0-1$.

If $\gamma_0\geq 2$, then by Eq.~(\mref{eq:v1}), $D_{v_1}$ is a linear combination of
$D_{v_i}, 2\leq i\leq \gamma_0$. Thus
$\rk(G)\leq\gamma_0-1$.
So we just need to prove $\rk(C_{va})\geq \gamma_0-1$. We will show this by induction on $\gamma_0\geq 2$.

First assume $\gamma_0=2$. Since $\Gamma$ is connected, by Eq.~(\mref{eq:vp}) the rows of $C_{va}$ are
non-zero. Then $\rk(C_{va})\geq 1$, as needed.

Next assume that the statement holds when $\gamma_0=n$ for $n\geq 2$
and consider a connected acyclic quiver $\Gamma$ with
$\gamma_0=n+1$. By first listing the arrows $p_1,\cdots,p_r$ of
$\Gamma$ that are connected to $v_1$, we see that $C_{va}$ is a
block matrix of the form
\begin{equation}
C_{va}=\left [ \begin{array}{cc} \vec{e} & \vec{0} \\ B & \overline{G} \end{array} \right ],
\notag %\mlabel{eq:indg}
\end{equation}
where $\vec{e}$ is a row vector of dimension $r$ with entries $\pm
1$, $\vec{0}$ is a zero row vector of dimension $\gamma_1-r$, $B$ is
a $(\gamma_0\times r)$-matrix and $\overline{G}$ is in fact the
vertex-arrow matrix $G_{\overline{\Gamma}}$ of the quiver
$\overline{\Gamma}$ obtained by deleting the vertex $v_1$ and its
attached arrows from $\Gamma$. Since $\overline{\Gamma}$ has $n$
vertices, by the induction hypothesis, the rank of $\overline{G}$ is
at least $n-1$. Thus there is a non-singular submatrix
$\overline{H}$ of $\overline{G}$ of size $(n-1)\times (n-1)$. Adding
back the first row and first column of $G$ to this submatrix
$\overline{H}$, we obtain a submatrix $H$ of $C_{va}$ of size
$n\times n$. Since the added first row is $(\pm 1,0,\cdots,0)$, the
added first column is not a linear combination of the other columns
in $H$. Thus $H$ is non-singular and the rank of $C_{va}$ is at
least $n$. This completes the induction and hence the proof of
Item~(\mref{it:a}).
\medskip

\noindent
(\mref{it:b})
We prove by induction on $\gamma_2$, which is also the number of
rows of $C_{ca}$.

When $\gamma_2=1$, $\pcyc_0$ is the unique primitive cycle. In this
case, $\Gamma$ is topologically homeomorphic to a point in the
surface $S$. By Remark~\mref{rk:cycle}(d), we have $D_{\pcyc_1}=0$.
Thus,  the unique row of $C_{ca}$ is zero. So,
$\rk(C_{ca})=0=\gamma_2-1$.

When $\gamma_2=2$, the rows of $C_{ca}$ are non-zero by definition
and the sum of the only two rows of $C_{ca}$ is zero by Eq.~(\mref{eq:zero}). Hence, $\rk(C_{ca})=1=\gamma_2-1$.

Assume that the statement is verified when $\gamma_2=n\geq 2$ and
consider $\Gamma$ with $\gamma_2=n+1$. By reordering the arrows
$p_1,\cdots,p_{\gamma_1}$ of $\Gamma$ if necessary, we can assume that the arrows of $\pcyc_1$ are $p_1,\cdots,p_s$ where $s\geq 2$ since $\Gamma$ is acyclic. Thus the coefficients
$c_{\gamma_0+1,1},c_{\gamma_0+1,2}, \cdots,c_{\gamma_0+1,\gamma_1}$
of $\pcyc_1$ in
Eq.~(\mref{eq:pp}) satisfy that
$c_{\gamma_0+1,1},\cdots,c_{\gamma_0+1,s}$ are all $\pm 1$ and
$c_{\gamma_0+1,s+1}=\cdots=c_{\gamma_0+1,\gamma_1}=0$. Then, since the sum of the row vectors of $C_{ca}$ is
zero by Eq.~(\mref{eq:zero}),  by reordering the primitive cycles
$\pcyc_1,\cdots,\pcyc_{n}$ if necessary, we can also assume that the
first coefficient of $\pcyc_1$ satisfies
$c_{\gamma_0+2,1}=-c_{\gamma_0+1,1}$. Thus, there is $r$ between $1$
and $s$ such that $c_{\gamma_0+2,i}=-c_{\gamma_0+1,i}$ for
$i=1,\cdots,r$ and $c_{\gamma_0+2,j}=0$ for $j=r+1,\cdots,s$. Then
the first two rows of $C_{ca}$ are of the form
$$\left [ \begin{array}{ccc}
\pm 1, \cdots, \pm 1, & \pm 1,\cdots,\pm 1, & 0,\cdots,0\\
\underbrace{\mp 1, \cdots, \mp 1}_{ r \text{ terms}},&
\underbrace{0,\cdots,0}_{s-r \text{ terms}}, & \underbrace{\ast,
\cdots, \ast}_{\gamma_1-r \text{ terms}}
\end{array} \right ]$$
Here the signs in the second row for the first $r$ terms are
opposite to the signs in the corresponding terms in the first row.
 Thus, the matrix $C_{ca}$ has the form
$$ C_{ca}=\left [ \begin{array}{ccc} \vec a_1 & \vec b_1 \\
-\vec a_1 & \vec b_2 \\ O & \overline{K} \end{array} \right ],$$
where $\vec{a}_1$ is a row vector of dimension $r$ with entries $\pm
1$, $O$ is a zero matrix of size $(n-1)\times r$, $\vec{b}_1$,
$\vec{b}_2$ are row vectors of dimension $\gamma_1-r$ and
$\overline{K}$ is a matrix of size $(n-1)\times (\gamma_1-r)$. We
can write $$\vec
b_1=(c_{\gamma_0+1,r+1}\;\cdots\;c_{\gamma_0+1,s}~\;0\;\cdots\;0)=(\pm
1\;\cdots\;\pm 1\;\;\;0\;\cdots\;0)$$
$$\vec
b_2=(0\;\cdots\;0~\;c_{\gamma_0+2,s+1}\;\cdots\;c_{\gamma_0+2,\gamma_1})=
(\;\;0\;\;\cdots\;\;\;0\;\;\ast\;\cdots\;\ast)$$

Deleting the arrows $p_1,\cdots,p_r$, we get a quiver
$\overline{\Gamma}$ with $n$ primitive cycles
$\pcyc_1',\pcyc_3,\cdots,\pcyc_{n}$ where $\pcyc_1'$ is obtained
via amalgamating $\pcyc_1$ and $\pcyc_2$. Moreover, with
 $$\vec
b:=\vec b_1+\vec
b_2=(c_{\gamma_0+1,r+1}\;\cdots\;c_{\gamma_0+1,s}\;c_{\gamma_0+2,s+1}\;\cdots\;c_{\gamma_0+2,\gamma_1}),$$
we find that the cycle-arrow matrix $C_{ca,\overline{\Gamma}}$ of the
quiver $\overline{\Gamma}$ is just
$\left [ \begin{array}{cc}  \vec b \\
 \overline{K} \end{array} \right ].$

By the induction hypothesis, $\rk(C_{ca,\overline{\Gamma}})=n-1$. By
Eq.~(\mref{eq:zero}), $-\vec b$ is the sum of all rows of $\overline
K$. Hence, $\rk(\overline K)=n-1$, that is, $\overline K$ has full
row rank. Since $r\geq 1$, the second row $(-\vec e_1 \; \vec b_2)$
of $C_{ca}$ is linearly independent from the last $n-1$ rows of
$C_{ca}$. So, $\rk(C_{ca})\geq (n-1)+1=n$. By
Eq.~(\ref{eq:zero}), the row vectors of $C_{ca}$ are linearly
dependent. Therefore, $\rk(C_{ca})=n=\gamma_2-1$, completing the
induction.

\medskip

\noindent
(\mref{it:c})
By definition, $B_{\Gamma}$ is a $(\gamma_2\times \gamma_2)$-matrix.

In the special case of $\gamma_2=1$, there are no cycles on $\Gamma$
except the boundary $\pcyc_1$ for the unique face. So all the
arrows are on $\pcyc_1$ only. Thus $B_\Gamma=0$ and hence its rank
is $0$.

For the case when $\gamma_2>1$, we apply Lemma~\mref{lem:rk}. For
this we just need to verify that $B_{\Gamma}$ satisfies the
conditions for $M$ in the lemma as follows.

\begin{itemize}
\item
By definition, $B_\Gamma$ has entries in $\ZZ$ and hence in a number field, say $\QQ$.
\item
By Definition~\mref{de:boundary} on $e_{i,j}$,  $B_\Gamma$ is
symmetric and $e_{i,j}\leq 0$ for $i\neq j$.
\item
$B_\Gamma$ is irreducible since $\Gamma$ is connected.
\item
Since $\gamma_0>0$, each cycle $\pcyc_j, 1\leq j\leq
\gamma_2$ shares at least one arrow with the cycle of a
neighboring face. Thus $e_{j,j}>0$.
\item
By the definition of $e_{j,r}$, $1\leq j,r\leq \gamma_2$, the sum of the entries of each row of $B_\Gamma$ is zero. Thus the sum of the columns of $B_\Gamma$ is zero. Since $B_\Gamma$ is symmetric, the sum of the rows of $B_\Gamma$ is also zero.
\end{itemize}
Thus Item~(\mref{it:c}) is proved.  \end{proof}

\subsection{Dimensions of combinatorial derivations}

\begin{theorem}
Let $\Gamma$ be a connected finite acyclic quiver. Then,
\begin{enumerate}
\item
$\dim \frakD_V=|V|-1$; \mlabel{it:A}
\item
$\dim \frakD_F=|F|-1$; \mlabel{it:B}
\item
$\frakD_V$ and $\frakD_F$ are linearly disjoint subspaces of
$\frakD_E$. \mlabel{it:C}
\end{enumerate}
\mlabel{thm:diffsum}
\end{theorem}

The following is a direct consequence of Theorem~\mref{thm:diffsum}. More applications of the theorem will be given in Section~\mref{sec:app}.
\begin{coro}
The dimensions of the spaces of derivations $\frakD_V$,
$\frakD_E$ and $\frakD_F$ of a connected acyclic quiver
 $\Gamma$ only depend on the underlying graph of $\Gamma$ and not depend on the choice of orientations of the edges.
\mlabel{co:deuler}
\end{coro}

\begin{proof}
Item~(\mref{it:A}) follow directly from of Theorem~\mref{thm:rank}.(\mref{it:a}) by Eq.~(\mref{eq:vp}). Similarly, Item~(\mref{it:B}) follows directly from Theorem~\mref{thm:rank}.(\mref{it:b}) by Eq.~(\mref{eq:pp}). So we just need to prove Item~\mref{it:C}.

Let $D$ be in $\frakDv\cap \frakDp$. Then there are constants $a_i,
1\leq i\leq \gamma_0$, and $b_j, 1\leq j\leq \gamma_2$, in
$\bfk$ such that
\begin{equation} \sum_{i=1}^{\gamma_0}
a_iD_{v_i} = D=\sum_{j=1}^{\gamma_2} b_j D_{\pcyc_j}.
\mlabel{eq:disj1}
\end{equation}
To prove Proposition~\mref{thm:diffsum}.(\mref{it:C}), we just need to show $D=0$.

We first rewrite the left sum of Eq.~(\mref{eq:disj1}). By Eq.~(\mref{eq:vp}), we obtain
\begin{equation}
\sum_{i=1}^{\gamma_0} a_iD_{v_i} = \sum_{i=1}^{\gamma_0} a_i
\sum_{k=1}^{\gamma_1} c_{i,k} D_{p_k,p_k} =\sum_{k=1}^{\gamma_1}
\big(\sum_{i=1}^{\gamma_0}c_{i,k}a_i\big) D_{p_k,p_k}.
\notag %\mlabel{sixtysix}
\end{equation}
Since by definition,
$c_{i,k}=\left \{\begin{array}{ll} 1, & t(p_k)=v_i, \\
-1, & h(p_k)=v_i, \\ 0, & \text{otherwise}, \end{array} \right. $ we
obtain
\begin{equation}
\sum_{i=1}^{\gamma_0} a_i D_{v_i} = \sum_{k=1}^{\gamma_1}
(a_{t(p_k)}-a_{h(p_k)})D_{p_k,p_k}. \mlabel{eq:disj2a}
\end{equation}

On the other hand, by Eq.~(\mref{eq:pp}), the right sum of
Eq.~(\mref{eq:disj1}) can be written as
\begin{equation}
\sum_{j=1}^{\gamma_2} b_j D_{\pcyc_j} = \sum_{j=1}^{\gamma_2}
b_j \left(\sum_{k=1}^{\gamma_1} d_{j,k}D_{p_k,p_k}\right)
=\sum_{k=1}^{\gamma_1}
\left(\sum_{j=1}^{\gamma_2}b_jd_{j,k}\right) D_{p_k,p_k}.
\mlabel{eq:disjr}
\end{equation}
Here by Eq.~(\mref{eq:primd}), $ d_{j,k}=\left \{\begin{array}{ll}
1, &\text{if } p_k \text{ is on } \pcyc_j \text{ in clockwise direction}, \\
-1, &\text{if } p_k \text{ is on } \pcyc_j \text{ in counter clockwise direction}, \\
0, &\text{if } p_k \text{ is not on } \pcyc_j.
\end{array} \right .
$

By Remark~\mref{rk:cycle}.(\mref{it:ppm}), any given arrow $p_k$ appears twice on the boundaries of the cycles $\pcyc_j$, $1\leq j\leq \gamma_2$, once in clockwise direction and once in
counter clockwise direction.
Let $1\leq x(p_k)\leq \gamma_2$ (resp. $1\leq y(p_k)\leq
\gamma_2$) denote the label of the cycle containing $p_k$ in the
clockwise (resp. counter clockwise) direction. It is possible that $x(p_k)=y(p_k)$. Then we have
$d_{x(p_k),k}=1$ and $d_{y(p_k),k}=-1$. Then Eq.~(\ref{eq:disjr}) becomes
\begin{equation}
\sum_{j=1}^{\gamma_2} b_jD_{\pcyc_j} =
\sum_{k=1}^{\gamma_1} (b_{x(p_k)}-b_{y(p_k)})D_{p_k,p_k}.
\mlabel{eq:disj2b}
\end{equation}

Thus by Eq.~(\mref{eq:disj2a}) and (\mref{eq:disj2b}), we find that
Eq.~(\mref{eq:disj1}) is equivalent to
\begin{equation}
D=\sum_{k=1}^{\gamma_1} (a_{t(p_k)}-a_{h(p_k)})D_{p_k,p_k} =
\sum_{k=1}^{\gamma_1} (b_{x(p_k)}-b_{y(p_k)})D_{p_k,p_k}. \mlabel{eq:disj2}
\end{equation}

Since the set $\{D_{p_k,p_k}\,|\, 1\leq k\leq \gamma_1\}$ of derivations is linearly independent, we thus obtain the following system of linear equations:
\begin{equation}
 a_{t(p_k)}-a_{h(p_k)} =
 b_{x(p_k)}-b_{y(p_k)}, \quad
1\leq k\leq \gamma_1. \mlabel{eq:disj3}
\end{equation}

\begin{lemma}
For $1\leq k\leq \gamma_1$, a change of the direction of $p_k$ leads to a change of signs on both sides of the $k$-th equation in Eq.~(\mref{eq:disj3}), yielding an equivalent equation.
\mlabel{lem:disj1}
\end{lemma}

\begin{proof} Let $\Gamma'$ be the quiver with $p_k$ replaced by $q_k$ in the opposite direction. Then $t(p_k)=h(q_k)$ and $h(p_k)=t(q_k)$. Then the left hand side of Eq.~(\mref{eq:disj3}) becomes
$$a_{h(q_k)}-a_{t(q_k)}=-(a_{t(q_k)}-a_{h(q_k)}).$$

Likewise, by the definition of $x(p_k)$ and $y(p_k)$ we have
$x(q_k)=y(p_k)$ and $y(q_k)=x(p_k)$. This reverses the sign on the right hand side of Eq.~(\mref{eq:disj3}).
\end{proof}

\begin{lemma} A set of solutions $b_j, 1\leq j\leq \gamma_2$ in the system Eq.~(\mref{eq:disj3}) satisfy the following system of linear equations:
\begin{equation}
e_{j,j} b_j + \sum_{1\leq r\leq \gamma_2, r\neq j} e_{j,r}
b_{r} =0, \quad 1\leq j\leq \gamma_2,
\mlabel{eq:disj4}
\end{equation}
where $e_{i,j}$ is defined in Definition~\mref{de:boundary}.
Thus the coefficients matrix of the system in Eq.~(\mref{eq:disj3}) is the \bounmat $B_\Gamma$ of $\Gamma$.
\mlabel{lem:disj2}
\end{lemma}

\begin{proof}
Fix a $1\leq j\leq \gamma_2$. Let $p_{k_1},\cdots,
p_{k_s}$ be the arrows on $\pcyc_j$ that are also on $\pcyc_r$ for some $r\neq j$,
that is, these are the arrows on $\pcyc_j$ that are not in the interior of
$\pcyc_j$. Because of Lemma~\mref{lem:disj1}, we can change the
directions of some of the arrows $p_{k_1},\cdots, p_{k_s}$ so that
all the arrows in $\pcyc_j$ go clockwise, without changing the
solution set of the system in Eq.~(\mref{eq:disj3}). After this is done, when we add the
$k_\ell$-th equations
$$ a_{t(p_{k_\ell})} - a_{h(p_{k_\ell})} = b_{x(p_{k_\ell})} - b_{y(p_{k_\ell})}, \quad 1\leq \ell\leq s,$$
in the system~(\mref{eq:disj3}), the left hand sides add up to zero.
For the right hand side, we have $x(p_{k_\ell})=j$, $1 \leq \ell\leq
s$, and $y(p_{k_\ell})$ is the label of the other cycle that
$p_{k_\ell}$ is on. Note that $s=e_{j,j}$ by definition.
Thus the right hand sides add up to
$$
e_{j,j} b_j + \sum_{1\leq r\leq \gamma_2, r\neq j} e_{j,r} b_r,$$
for $e_{j,j}$ and $e_{j,r}$ as defined in the lemma. This finishes the proof.
\end{proof}

We can now complete the proof of Theorem~\mref{thm:diffsum}.(\mref{it:C}) as
follows. Let $D$ be in $\frakDv\cap \frakD_F$. Then there are $a_i, 1\leq
i\leq \gamma_0$ and $b_j, 1\leq j\leq\gamma_2$ such that
they satisfy Eq.~(\mref{eq:disj1}). Then they satisfy the linear
system in Eq.~(\mref{eq:disj3}).
Then by Lemma~\mref{lem:disj2}, $b_j$
satisfies the system of linear equations in  Eq.~(\mref{eq:disj4}).
Since the coefficient matrix of this system is the \bounmat $B_\Gamma$ by Lemma~\mref{lem:disj2} and hence has rank $\gamma_2-1$ by Theorem~\mref{thm:rank}.(\mref{it:c}), the system in Eq.~(\mref{eq:disj4}) has unique nonzero solutions $b_i$ up a constant. But the choice of $b_{x(p_k)}=b_{y(p_k)}, 1\leq k\leq \gamma_1$
together with $a_{t(p_k)}=a_{h(p_k)}, 1\leq i\leq \gamma_0$ is
already a nonzero solution of Eq.~(\mref{eq:disj3}) and hence of
Eq.~(\mref{eq:disj4}). Thus this gives the unique solution of
Eq.~(\mref{eq:disj3}).
%So the only solutions of the linear system Eq.~(\mref{eq:disj3}) satisfy
%\begin{equation}  a_{t(p_k)}=a_{h(p_k)}, \quad
% b_{x(p_k)}=b_{y(p_k)}, \quad 1\leq k\leq \gamma_1. \mlabel{eq:soln2} \end{equation}
Hence by Eq.~(\ref{eq:disj2}),
we have $D=0$. Therefore $\frakDv\cap \frakD_F=0$, showing that
$\frakDv$ and $\frakD_F$ are linearly disjoint.

Now the proof of Theorem~\mref{thm:diffsum}(c)
 is completed.
\end{proof}

\section{Euler's formula and Hochschild cohomology}
\mlabel{sec:app}
In this section we give two applications of the dimensional formulas of combinatorial derivations in Theorem~\mref{thm:diffsum}. We first present a differential strengthening of Euler's Polyhedron Theorem (Theorem~\mref{thm:euler}) by showing that the numerical relation in Euler's formula among the geometric objects of vertices, edges and faces of the underlying graph of a quiver comes from an algebraic relation among the spaces of derivations associated to these geometric objects.

\subsection{Euler's polyhedron formula from a differential point of view}
\mlabel{ss:deuler}

%\subsubsection{Euler's formula}

We have the following classical result~\mcite{Bo}\mcite{GT}.

\begin{theorem} $(${\bf Euler's Polyhedron Theorem}$)$ For any connected quiver $\Gamma$ of genus $g$,
we have
\begin{equation}
 |V|-|E|+|F|=2-2g.
\mlabel{eq:euler}
\end{equation}
\mlabel{thm:euler}
\end{theorem}

We show that there is a strengthening of Euler's numerical formula
in the context of derivations.
\begin{theorem} ({\bf Differential Formulation of Euler's Polyhedron Theorem})
For a connected finite acyclic quiver $\Gamma$ of genus $g$, the spaces $\frakD_V, \frakD_E, \frakD_F$ of vertex
derivations, arrow derivations and face derivations satisfy the following relation.
\begin{equation}
    \dim_{\bfk}\frakD_E/(\frakD_V\oplus \frakD_F)=2g,
\mlabel{eq:dimdiffsum}
\end{equation}
where the direct sum is the interior sum of subspaces.

In particular, in the case the genus $g=0$,
\begin{equation}
    \frakD_E = \frakD_V\oplus \frakD_F.
\mlabel{eq:diffsum}
\end{equation}
\mlabel{thm:diffeuler}
\end{theorem}
\begin{proof}
By Theorem~\mref{thm:diffsum}, we have
$$\dim(\frakD_V \oplus \frakD_F)=\dim \frakD_V +\dim \frakD_F=|V|-1+|F|-1.$$
But, $ \dim \frakD_E = |E|.$ By Euler's formula, \begin{equation}
\dim \frakD_E-\dim(\frakD_V \oplus \frakD_F)= |E|-(|V|-1+|F|-1)=2g.
\mlabel{eq:proofdiffsum}
\end{equation}
 This gives
Eq.~(\mref{eq:dimdiffsum}) since $\frakD_V\oplus \frakD_F$ is a
subspace of $\frakD_E$.
\end{proof}

\begin{remark}
{\rm As we see above, Euler's theorem is used in the proof of
Theorem~\mref{thm:diffeuler}. Conversely, the equation
$\dim_{\bfk}\frakD_E/(\frakD_V\oplus \frakD_F)=2g$ gives Euler's
formula by Eq.~(\mref{eq:proofdiffsum}). Thus, as its name suggests,
Theorem~\mref{thm:diffeuler} gives a strengthened form of Euler's
theorem from the view point of derivations. It would be interesting to find a proof of Theorem~\mref{thm:diffeuler} without using Euler's theorem.
%Taking dimensions in the direct sum, we obtain $$ \dim \frakD_V - \dim \frakD_E +\dim \frakD_F=0.$$ Since $$ \dim \frakD_E=|E|,$$ we immediately get Euler's theorem.
}
\end{remark}

\subsection{Outer derivations on a path algebra}
\mlabel{ss:out}

\delete{
Consider the exact sequence
\begin{equation}
0\to \indiff(\bfk\Gamma) \to \diff(\bfk\Gamma) \to HH^1(\bfk\Gamma)
\to 0
\notag %\mlabel{eq:seq}
\end{equation}
from Eq.~(\mref{eq:outd}).

By Corollary~\mref{co:qcent}, we have
$\bfk\Gamma/\bfk\cong \indiff(\bfk\Gamma)$ unless $\Gamma$ is an
oriented cycle. Thus,  to a large extent, $\indiff(\bfk\Gamma)$
determines the algebraic properties of $\bfk\Gamma$.
}

We next apply Theorem~\mref{thm:diffsum} to study $HH^1(\bfk\Gamma)$. We first
give a dimensional formula of $HH^1(\bfk\Gamma)$. Then we obtain a canonical basis
of $HH^1(\bfk\Gamma)$.

\subsubsection{A dimensional formula of $HH^1(\bfk\Gamma)$}

Denote
$$ \calpc:=\{s\in \calp\,|\, h(s)=t(s)\}, \quad \frakBc:=\{D_s\ |\ s\in \calpc\}, \quad \frakDc:=\bfk\frakBc=\frakD(\bfk \calpc).$$
Then on one hand we have the disjoint union $ \calp=\calpa \sqcup
\calpc,$ and hence $\bfk \Gamma = \bfk \calpa \oplus \bfk \calpc.$
Since $\ker \frakD=\bfk\subseteq \bfk \calpc$ for the linear map
$\frakD:\bfk\Gamma\to \diff(\bfk\Gamma)$ in Eq.~(\mref{eq:indh}), we
have
\begin{equation}
 \indiff(\bfk\Gamma) =\frakD(\bfk \Gamma) \cong \frakDa \oplus \frakDc.
\mlabel{eq:ind2}
\end{equation}

On the other hand, since $\frakDa\subseteq \indiff(\bfk\Gamma)$ and
$\diff(\bfk\Gamma)=\frakDa\oplus \frakDal$ from
Theorem~\mref{thm:dbase}, we have
\begin{equation}
\indiff(\bfk\Gamma) =\indiff(\bfk\Gamma)\cap(\frakDa\oplus
\frakDal)=\frakDa \oplus (\indiff(\bfk\Gamma)\cap \frakDal).
\mlabel{eq:ind3}
\end{equation}
By Proposition~\mref{pp:basis2}, we have $\frakDc \subseteq
\indiff(\bfk\Gamma)\cap\frakDal$. Thus from Eq.~(\mref{eq:ind2}) and
Eq.~(\mref{eq:ind3}) we obtain
\begin{equation}
\indiff(\bfk\Gamma)\cap\frakDal =\frakDc.
\notag %\mlabel{eq:ind4}
\end{equation}
Therefore we have
\begin{eqnarray}
HH^1({\bf k}\Gamma)&=&\diff(\bfk \Gamma)/\indiff(\bfk\Gamma) \notag \\
&=& (\indiff(\bfk\Gamma)+\frakDal)/\indiff(\bfk\Gamma) \notag \\
&\cong& \frakDal/(\indiff(\bfk\Gamma)\cap \frakDal)\mlabel{eq:outdiff}\\
& =&
\frakDal/\frakDc,\notag
\end{eqnarray}
giving us the following commutative diagram of exact sequences
of Lie algebras.
$$
\begin{CD}
     0 @>>> \indiff({\bf k}\Gamma) @>>> \diff({\bf k}\Gamma) @>>> HH^1({\bf k}\Gamma) @>>> 0\\
                         @.          @AAA                   @AAA                             @|                          @. \\
 0 @>>>     \frakDc     @>>>      \frakDal        @>>>       \frakDal/\frakDc               @>>> 0
 \end{CD}
$$

An {\bf almost oriented cycle} in a quiver $\Gamma$ is defined to be
a pair $(p,r)$ where $p\in E$ and $r\in \calp$ with $r\not=p$ and $r\parallel p$. It is so namely since $p$ and $r$ form an oriented cycle by reversing the arrow $p$. Let $\calpal$ be the set of almost
oriented cycles of $\Gamma$. Denote
\begin{equation}
\frakB_E:=\{D_{p,p}\ | p\in E\},\quad
   \frakB_{AL}:=\{D_{p,r}\ | (p,r)\in\Gamma_{AL}\},\quad
\frakD_{AL}=\bfk \frakB_{2,2}. \mlabel{eq:diff22}
\end{equation}
Then by the definition of $\frakBal$ in Eq.~(\mref{eq:baal}), we have the disjoint union
\begin{equation} \frakBal = \frakB_E \sqcup \frakB_{AL} \
\ \ \ \text{and} \ \ \ \ \frakDal = \frakD_E \oplus \frakD_{AL}.
\mlabel{eq:b2}
\end{equation}
From this and Eq.(\ref{eq:outdiff}),  we obtain
\begin{equation}
HH^1(\bfk\Gamma) \cong \frakDal/\frakDv \cong
\big(\frakD_E/\frakDv\big) \oplus \frakD_{AL}. \mlabel{eq:gamma2}
\end{equation}

\begin{prop}
Let $\Gamma$ be a connected  acyclic quiver. Then
$$\dim_\bfk HH^1(\bfk\Gamma) = |F| +|\calpal|-1+2g. $$
In the case $g=0$, $HH^1(\bfk\Gamma)\not=0$ if and only if\,
$\Gamma$ contains an unoriented cycle. \mlabel{pp:odim}
\end{prop}

\begin{proof}
By Eq.~(\mref{eq:gamma2}),  Theorem~\mref{thm:rank} and Euler's theorem, we have
$$\dim HH^1(\bfk\Gamma)=\dim \frakD_E-\dim \frakDv+\dim\frakD_{AL} =|E|-|V|+1+|\Gamma_{AL}| = |F|+|\Gamma_{AL}|+2g-1.$$

Then the last statement follows easily.
\end{proof}

Compare Proposition~\mref{pp:odim} with Happel's formula~\cite{Ha},
\begin{equation}\label{eq:dim1}
\dim_{\bf k}HH^1({\bf k}\Gamma)=1-\mid V\mid+\sum_{\alpha\in
E}v(\alpha)
\end{equation}
where $v(\alpha)=\dim_{\bf k}t(\alpha){\bf k}\Gamma h(\alpha)$.
The two formulas can be easily derived from each other. Our formula makes it easy to guess a canonical basis of $HH^1(\bfk\Gamma)$. Indeed, verifying that this guess actually works is the motivation behind the introduction of the combinatorial derivations in Section~\mref{sec:comb}.

\subsubsection{A canonical basis of $HH^1(\bfk \Gamma)$}
\mlabel{ss:obase}
Proposition~\mref{pp:odim} suggests that a
canonical basis of $HH^1(\bfk\Gamma)$ can be obtained from
derivations defined from the faces (through their unoriented cycles) and $\calpal$. We show that this is indeed the case.

Denote
\begin{equation}
\frakB_F^-=\{D_{\pcyc_i}\ |\ 2\leq i\leq \gamma_2\}, \mlabel{eq:diffp}
\end{equation}
then by Theorem~\mref{thm:diffsum}.(\mref{it:B}) and Eq.~(\mref{eq:zero}), we have  $\frakD_F=\bfk\frakB_F^-$.
Our main result on outer derivations is the following

\begin{theorem} {\bf (Basis Theorem of the First Hochschild Cohomology)}
Let $\Gamma$ be a connected finite acyclic quiver over a ground
field $\bfk$ of characteristic $0$. Then the disjoint union
$$\frakB_{AL}\sqcup \frakB_F^-\sqcup\frakB^*$$
 forms a basis of $HH^1(\bfk\Gamma)$, where $\frakB^*$ is any basis of $\frakD_E/(\frakD_V\oplus\frakD_F)$. \mlabel{thm:obase}
\end{theorem}
\begin{proof}
We have
$\frakD_E/(\frakD_V\oplus\frakD_F)\cong(\frakD_E/\frakD_V)/((\frakD_V\oplus\frakD_F)/\frakD_V)$.
It follows that
\begin{eqnarray*}
\frakD_E/\frakD_V &\cong & \frakD_E/(\frakD_V\oplus\frakD_F)\oplus
(\frakD_V\oplus\frakD_F)/\frakD_V\\
 &\cong &
\frakD_E/(\frakD_V\oplus\frakD_F)\oplus\frakD_F.
\end{eqnarray*}
 Thus by Eq.~(\mref{eq:gamma2}),
$HH^1(\bfk\Gamma)\cong (\frakD_E/\frakDv)\oplus \frakD_{AL}\cong
\frakD_E/(\frakD_V\oplus\frakD_F)\oplus\frakD_F\oplus \frakD_{AL}.$

  $\frakD_F$
has a basis $\frakB_F^-$ by
Theorem~\mref{thm:diffsum}.(\mref{it:B}) and $\frakD_{AL}$ has a
basis $\frakB_{AL}$. Since they are disjoint by Eq.~(\mref{eq:b2}), we
obtain Theorem~\mref{thm:obase} for any basis $\frakB^*$ of
$\frakD_E/(\frakD_V\oplus\frakD_F)$.
\end{proof}

We require that $\bfk$ is of characteristic $0$ in
Theorem~\mref{thm:obase} because its proof depends on
Theorem~\mref{thm:lt}.
When the genus $g$ is $0$, $\Gamma$ is a planar quiver. In this case, by Euler's theorem, $\dim\frakD_E/(\frakD_V\oplus\frakD_F)=0$. Also, we can take $\pcyc_1$ to be the primitive cycle for the unique unbounded face. Thus we can make the choice of $\frakB_F^-$ as well as $\frakB_{AL}$ completely canonical.
Hence, we have
\begin{coro} {\bf (Basis Theorem for Planar Quiver)}
Let $\Gamma$ be a connected finite planar acyclic quiver and let the
ground field $\bfk$ be of characteristic $0$. Then the disjoint
union
$$\frakB_{AL}\cup \{D_\pcyc \,|\,\pcyc \text{ is a bounded primitive cycle}\}$$
 forms a basis of $HH^1(\bfk\Gamma)$. \mlabel{coro:obase}
\end{coro}

\subsubsection{Lie algebra structure of $HH^1(\bfk\Gamma)$}
From the basis $\frakB_{AL}\cup \frakB_F^-$ in
Corollary~\mref{coro:obase}, the structural constants of the Lie algebra
$HH^1(\bfk\Gamma)$ can be computed explicitly by Eq.~(\mref{eq:lie3}) since we
have $\frakB_{AL}\subseteq \frakB_2$ and $\frakB_F\subseteq
\bfk\frakB_{E}=\frakD_E\subseteq \frakD_2=\bfk\frakB_2$. Concretely,
for $D_{\pcyc}, D_{\mathfrak q}\in \frakB_F$, let
\[ D_{\pcyc}= a_1 D_{p_1,p_1}+a_2D_{p_2,p_2}+
\cdots +a_{\gamma_1} D_{p_{\gamma_1},p_{\gamma_1}}, \;\;\;\;\;\;\;
D_{\mathfrak q}= b_1 D_{p_1,p_1}+b_2D_{p_2,p_2}+ \cdots
+b_{\gamma_1} D_{p_{\gamma_1},p_{\gamma_1}}
\]
for the arrow set $E=\{p_1,p_2,\cdots,p_{\gamma_1}\}$. Then by
Eq.~(\mref{eq:lie3}) we have \begin{eqnarray*}
[D_{\pcyc},D_{\mathfrak q}]&=&
\sum_{i,j=1}^{\gamma_1}a_ib_j[D_{p_i,p_i},D_{p_j,p_j}] \\
&=&\sum_{i,j=1}^{\gamma_1}a_ib_j(D_{p_j,D_{p_i,p_i}(p_j)}-D_{p_i,D_{p_j,p_j}(p_i)}) \\
&=&\sum_{i=1}^{\gamma_1}a_ib_i(D_{p_i,p_i}-D_{p_i,p_i}) \\
&=& 0.
\end{eqnarray*}
This means that $\frakD_F=\bfk\frakB_F$ is an abelian Lie sub-algebra
of the Lie algebra $HH^1(\bfk\Gamma)$.

For $D_{r,s}\in \frakB_{AL}$, by Eq.~(\mref{eq:lie3}) we have
\begin{eqnarray*}
[D_{r,s},D_{\pcyc}]&=& a_1[D_{r,s},D_{p_1,p_1}]+\cdots+a_{\gamma_1}[D_{r,s},D_{p_{\gamma_1},p_{\gamma_1}}] \\
&=& a_1(D_{p_1,D_{r,s}(p_1)}-D_{r,D_{p_1,p_1}(s)})+\cdots+a_{\gamma_1}(D_{p_{\gamma_1},D_{r,s}(p_{\gamma_1})}-D_{r,D_{p_{\gamma_1},p_{\gamma_1}}(s)}) \\
&=&
(a_1D_{p_1,D_{r,s}(p_1)}+\cdots+a_{\gamma_1}D_{p_{\gamma_1},D_{r,s}(p_{\gamma_1})})
-(a_1D_{r,D_{p_1,p_1}(s)}+\cdots+a_{\gamma_1}D_{r,D_{p_{\gamma_1},p_{\gamma_1}}(s)}).
\end{eqnarray*}
Write $r=p_{i_0}$ for some $1\leq i_0\leq \gamma_1$. Then
$a_1D_{p_1,D_{r,s}(p_1)}+\cdots+a_{\gamma_1}D_{p_{\gamma_1},D_{r,s}(p_{\gamma_1})}=a_{i_0}D_{r,s}$.
Also, write $s=p_{i_1}\cdots p_{i_t}$ for arrows $p_{i_1}, \cdots,
p_{i_t}$. Thus,
$a_1D_{r,D_{p_1,p_1}(s)}+\cdots+a_{\gamma_1}D_{r,D_{p_{\gamma_1},p_{\gamma_1}}(s)}=(a_{i_1}+\cdots+a_{i_t})D_{r,s}.$
In summary,
$$[D_{r,s},D_{\pcyc}] =(a_{i_0}-a_{i_1}-\cdots-a_{i_t})D_{r,s}.$$
which means that $D_{r,s}$ is the eigenvector under the adjoint
action of $D_\pcyc$ with eigenvalue $-a_{i_0}+a_{i_1}+\cdots
+a_{i_r}.$
 It follows that $\frakD_{AL}=\bfk\frakB_{AL}$ is a
Lie ideal of $HH^1(\bfk\Gamma)$. We have proved the following
result.
\begin{theorem}
{\bf (Structure Theorem of the First Hochschild Cohomology)} Let $\Gamma$ be a connected planar finite acyclic quiver and let the
ground field $\bfk$ be of characteristic $0$. Then the Lie algebra
$HH^1(\bfk\Gamma)$ is the semi-direct sum of the Lie ideal
$\frakD_{AL}$ and the abelian Lie subalgebra $\frakD_F$:
\[HH^1(\bfk\Gamma)=\frakD_{AL}\rtimes_\varphi \frakD_F.\]
Here the action $\varphi$ of $\frakD_F$ on $\frakD_{AL}$ is given as
follows. For each given $D_\pcyc$ with $\pcyc=a_1
D_{p_1,p_1}+a_2D_{p_2,p_2}+ \cdots +a_{\gamma_1}
D_{p_{\gamma_1},p_{\gamma_1}}$ and $D_{r,s}\in \frakD_{AL}$ with
$r=p_{i_0}$ and $s=p_{i_1}\cdots p_{i_t}$, $D_{r,s}$ is the
eigenvector under the adjoint action of $D_\pcyc$ with eigenvalue
$-a_{i_0}+a_{i_1}+\cdots +a_{i_r}.$
\mlabel{thm:str}
\end{theorem}

From this proposition, we conclude that when two quivers are defined on the same
unoriented graph, their outer differential Lie algebras are not
isomorphic in general unless they have the same set of almost oriented cycles.

\begin{remark}
{\rm We conclude our paper with a brief discussion on the
characteristic of the ground field $\bfk$ in
Theorem~\mref{thm:obase} and Corollary~\mref{coro:obase}. The
requirement that the characteristic of $\bfk$ is zero is because in
the proof of this theorem, we need Theorem~\mref{thm:lt} which
requires that the ground field is a number field. We believe that
Theorem~\mref{thm:obase} and Corollary~\mref{coro:obase} remain
true without this condition, but in that case the proof cannot be
obtained by applying Theorem~\mref{thm:lt}.

As an example, consider the quiver $\Gamma$ with vertex set
$V=\{v_1,v_2\}$ and arrow set $E=\{p_1,p_2\}$ where both $p_1$ and
$p_2$ go from $v_1$ to $v_2$. Then by Theorem~\mref{thm:dbase}, we
find that $ \diff(\bfk\Gamma)$ has a basis given by
$\{D_{p_1},D_{p_2},D_{p_1,p_1},
D_{p_2,p_2},D_{p_1,p_2},D_{p_2,p_1}\}.$ Also
$\indiff(\bfk\Gamma)=\bfk\{D_{v_1},D_{p_1},D_{p_2}\}.$ Since
$D_{v_1}=D_{p_1,p_1}+D_{p_2,p_2}$, we find that a basis of
$HH^1(\bfk\Gamma)$ is given by (the cosets of)
$\{D_{p_1,p_1}-D_{p_2,p_2},D_{p_1,p_2},D_{p_2,p_1}\}.$ Thus,
Corollary~\mref{coro:obase} holds for $\bfk\Gamma$ over any field
$\bfk$.

However, by a direct computation, we have $ C_\Gamma=\left (
\begin{array}{cc} 1&1\\-1&-1 \\1&1 \\-1&-1\end{array}\right )$ and
hence $ B_\Gamma=\left (\begin{array}{cc} 2 &-2 \\ -2 &
2\end{array}\right ).$

 Thus if the characteristic of $\bfk$ is 2,
then the rank of $B_\Gamma$ is zero. So the conclusion of
Theorem~\mref{thm:rank}.(\mref{it:c}) does not hold. Therefore we cannot apply
Theorem~\mref{thm:rank}.(\mref{it:c}) to prove Theorem~\mref{thm:obase} and
Corollary~\mref{coro:obase} for our $\Gamma$ even though we have
verified that the conclusions of Theorem~\mref{thm:obase} and
Corollary~\mref{coro:obase} hold for this quiver.
}
\end{remark}

{\bf Acknowledgements.} Li Guo acknowledges support by NSF grant DMS 1001855
and thanks the Center of Mathematical Sciences in Zhejiang
University for hospitality. Fang Li takes this opportunity to
express thanks to the support from the National Natural Science
Foundation of China (No.10871170) and the Zhejiang Provincial
Natural Science Foundation (No.D7080064). This
research was supported in part by the Project of Knowledge
Innovation Program (PKIP)
  of Chinese Academy of Sciences, Grant No. KJCX2.YW.W10.

\end{document}